\documentclass{article}
\usepackage{graphicx,amssymb,amsmath,amsthm,amsfonts,hyperref, tikz-cd, quiver} 
\newcommand{\linearspace}{{linear space}}
\newtheorem{lemma}{Lemma}
\newtheorem{definition}{Definition}
\newtheorem{proposition}{Proposition}

\newtheorem{example}{Example}

\title{Linear Spaces over Perfect Idylls}
\author{Jeffery Liu}
\date{June 25 2026}

\begin{document}
\maketitle
\begin{abstract}
    We construct a category of vector space-like objects called \emph{\linearspace{}s} over perfect idylls $k$ (an algebraic structure generalizing fields and hyperfields), which is a specialization of modules over $k$. Previous authors have noted that naive linear algebra in the product $k^n$ fails because linear independence does not satisfy matroid independence axioms. We give a sufficient axiom so that linear independence in $k$-\linearspace{}s does satisfy matroid independence axioms. We also examine basic categorical properties of $k$-\linearspace{}s. In particular, the category of $k$-\linearspace{}s has no products, explaining the failure of naive linear algebra in $k^n$. Furthermore, we explore the categorical relationship between \linearspace{}s over a perfect idyll, ordinary matroids, and matroids over a perfect idyll, connecting the existing theories of matroids and modules.
\end{abstract}
\section{Introduction}
Idylls (resp. bands) and hyperfields (resp. commutative hyperrings) are algebraic structures generalizing fields (resp. commutative rings). Notable examples are the Krasner hyperfield $\mathbb{K}$, the sign hyperfield $\mathbb{S}$, and the tropical hyperfield $\mathbb{T}$. Modules may be formed over a band or idyll, generalizing modules over a commutative ring or a vector space over a field.

For an arbitrary idyll $k$, naive linear algebra in modules over $k$ (in particular, in the product $k^n$) does not work well in general, because linear independence of its elements does not necessarily satisfy the matroid independence axioms. Instead, other authors have generalized linear algebra over an idyll in alternate ways. In \cite{BakerBowler19}, Baker and Bowler introduced matroids over an idyll $k$, or $k$-matroids, including their axiomatizations by $k$-valued Pl\"ucker functions/coordinates $\subset k^{\binom{n}{r}}$. These generalize the Pl\"ucker embedding of Grassmanian $\mathrm{Gr}(n,r)$ of dimension $r$ linear subspaces of $k^n$ when $k$ is a field. Baker and Bowler also defined the set of vectors, or $k$-vector sets and covectors for a $k$-matroid. These generalize the linear subspaces of $k^n$ when $k$ is a field. In \cite{Anderson19}, Anderson gives an independent equivalent axiomatization for $k$-vector sets.

Over the Krasner hyperfield ($k = \mathbb{K}$), $\mathbb{K}$-matroids are equivalent to ordinary matroids, since the $\mathbb{K}$-Pl\"ucker coordinates reduce to the bases of an ordinary matroid and the $\mathbb{K}$-vector sets reduce to the cycles of the matroid.

The case over the sign hyperfield ($k = \mathbb{S}$), was studied earlier by Bland and Las Vergnas in \cite{BlandVergnas78}, where $\mathbb{S}$-matroids are called oriented matroids. Their $\mathbb{S}$-Pl\"ucker coordinates are also called chirotopes.

The case over the tropical hyperfield ($k = \mathbb{T}$) was studied earlier by Dress and Wenzel in \cite{DressWenzel92}, where $\mathbb{T}$-matroids are called valuated matroids. Their $\mathbb{T}$-Pl\"ucker coordinates are also called basis valuation functions, and their $\mathbb{T}$-vector sets are called tropical linear spaces.

An idyll $k$ is called \emph{perfect} if, for any $k$-matroid, its set of vectors are orthogonal to the set of covectors. Vector sets over a perfect idyll have desirable restriction and contraction properties. The aforementioned idylls, $\mathbb{K}$, $\mathbb{S}$, and $\mathbb{T}$, as well as all fields, are perfect idylls.

Our contribution in this paper is the following. We construct a special class of modules over perfect idylls, called \emph{\linearspace{}s}, for which linear independence \textit{does} satisfy the matroid independence axioms. In particular, we show that the category of simple matroids with pointed strong maps is equivalent to the category of finitely generated \linearspace{}s over the Krasner hyperfield $\mathbb{K}$, and that the category of simple matroids over a perfect idyll $k$ with submonomial matrix maps may be faithfully embedded in the category of \linearspace{}s over $k$. The failure of linear independence in $k^n$ can be attributed to the fact that the category of \linearspace{}s does not have products, which is expected because the category of simple matroids also does not have products.

The paper is organized as follows: 
\begin{itemize}
    \item In Section \ref{sec:matroid_theory}, we recall the basic definitions and properties of matroids and their pointed strong maps.
    
    \item In Section \ref{sec:band_idyll}, we restate the theory of algebraic structures called bands (resp. idylls) and commutative hyperrings (resp. hyperfields), which generalize commutative rings (resp. fields), as it appears in \cite{BakerJinLorscheid25}.
    
    \item In Section \ref{sec:band_module}, we examine the construction of modules over a band (which includes modules over an idyll, commutative hyperring, or hyperfield).

    \item In Section \ref{sec:valuated_matroid}, we survey the current theory of $k$-matroids over an idyll $k$, including their axiomatizations by $k$-Pl\"ucker coordinates and $k$-vector sets.

    \item In Section \ref{sec:proto-exact}, we review a property of categories called proto-exactness which the category of matroids, $k$-matroids, and modules over a band satisfy.

    \item Our main new contributions are in Section \ref{sec:linear_space}, where we construct \emph{\linearspace{}s} over a perfect idyll $k$ as a particular type of modules over $k$. We show that linear independence in these \linearspace{}s satisfies the matroid independence axioms by its construction via $k$-vector sets. We also examine categorical properties of \linearspace{}s. In particular, the category of \linearspace{}s is proto-exact, and has equalizers, kernels, cokernels, coproducts, but no coequalizers and products in general. We also show that the simple matroids with pointed strong maps is equivalent to the category of finitely generated \linearspace{}s over the Krasner hyperfield $\mathbb{K}$, and that the category of simple matroids over a perfect idyll $k$ with submonomial maps can be faithfully embedded into the category of finitely generated \linearspace{}s over $k$.
\end{itemize}
Sections \ref{sec:matroid_theory}-\ref{sec:proto-exact} serve to fix notation, to recall necessary preliminaries, and to collect existing constructions from other sources which formalize linear algebra at this level of generality. Our contribution in Section \ref{sec:linear_space} is to construct, over a perfect idyll $k$, a particular class of objects over $k$...
\begin{itemize}
    \item which generalize vector spaces in a coordinate free way when $k$ is a field,
    \item and for which linear independence of elements satisfies matroid independence axioms.
\end{itemize}
Modules over $k$ do the former but not the latter. Vector sets of matroids over $k$ only partially satisfy both: they generalize linear subspaces of $k^n$ but rely on coordinates, and have a notion of linear independence for the coordinate vectors of $k^n$ but not for all elements. Our construction connects the theory of modules and matroids over $k$ by producing a particular class of modules for which vector arrangements realize all matroids over $k$, and it satisfies both of the above properties.

\section{Matroid Theory}\label{sec:matroid_theory}

A matroid is a combinatorial object which encodes independence and dependence among elements of a finite set. A standard reference is \textit{Theory of Matroids} \cite{White86}.
\begin{definition}\label{def:matroid_independence}
    Let $E$ be a finite set. A \textbf{matroid} on $E$ is characterized by a family of \textbf{independent sets} $\mathcal{I} \subset 2^E$, which satisfy
    \begin{itemize}
        \item $\emptyset \in \mathcal{I}$
        
        \item (\textit{downward-closure}) for $S_1, S_2 \subset E$, if $S_2 \in \mathcal{I}$ and $S_1\subset S_2$, then $S_1 \in \mathcal{I}$
        
        \item (\textit{independence augmentation}) for $S_1, S_2 \in \mathcal{I}$, if $|S_1| > |S_2|$, then there exists $i \in S_1 \setminus S_2$ such that $S_2 \cup \{i\} \in \mathcal{I}$
    \end{itemize}
    Given a matroid $M$ with independent sets $\mathcal{I} \subset 2^E$,
    \begin{itemize}
        \item a subset $S \subset E$ is called \textbf{independent} if $S \in \mathcal{I}$.
        \item a subset is called \textbf{dependent} if and only if it is not independent.
        \item a \textbf{basis} is an inclusion-wise maximal independent set.
        \item a \textbf{circuit} is an inclusion-wise minimal dependent set.
        \item a \textbf{cycle} or \textbf{cyclic set} is a union of circuits.
        \item the \textbf{rank} of a subset $S \subset E$ is the size of the largest independent set which it contains. The rank of the matroid is the rank of the whole set $E$.
        \item a subset is called \textbf{flat} or \textbf{closed} if it is inclusion-wise maximal among subsets with the same rank. The flats of a matroid form a geometric lattice under inclusion.
    \end{itemize}
\end{definition}

\begin{definition}\label{def:matroid_basis}
    A matroid can be equivalently characterized by its set of bases $\mathcal{B} \subset 2^E$, which satisfy
    \begin{itemize}
        \item $\mathcal{B} \neq \emptyset$
        \item (\textit{basis exchange}) for $B_1, B_2 \in \mathcal{B}$, for each $i \in B_1 \setminus B_2$ there exists $j \in B_2 \setminus B_1$ such that $B_1 \setminus\{i\} \cup \{j\} \in \mathcal{B}$
    \end{itemize}
    Then, a subset is independent if it is contained in a basis. In particular, the basis exchange property implies that all bases have the same size, equal to the rank of the matroid.
\end{definition}

\begin{definition}\label{def:matroid_circuit}
    A matroid can be equivalently characterized by its set of circuits $\mathcal{C} \subset 2^E$, which satisfy
    \begin{itemize}
        \item $\emptyset \notin \mathcal{C}$
        \item (\textit{incomparability}) for $C_1, C_2 \in \mathcal{C}$, if $C_1 \subset C_2$, then $C_1 = C_2$
        \item (\textit{circuit elimination}) for $C_1, C_2 \in \mathcal{C}$, with $C_1 \neq C_2$, and for $i \in C_1 \cap C_2$, there exists $C_3 \in \mathcal{C}$ with $C_3 \subset C_1 \cup C_2 \setminus \{i\}$
    \end{itemize}
    Then, a subset is independent if it does not contain a circuit.
\end{definition}

\begin{definition}\label{def:matroid_rank}
    A matroid can be equivalently characterized by its \textbf{rank function} $\mathrm{rk}: 2^E \rightarrow \mathbb{Z}^{\geq 0}$, which satisfies
    \begin{itemize}
        \item for $A \subset E$, $\mathrm{rk}(A) \leq |A|$
        \item (\textit{monotonicity}) for $A, B \subset E$, if $A \subset B$, then $\mathrm{rk}(A) \leq \mathrm{rk}(B)$
        \item (\textit{submodularity}) for $A, B \subset E$, $\mathrm{rk}(A\cup B) + \mathrm{rk}(A \cap B) \leq \mathrm{rk}(A)+ \mathrm{rk}(B)$
    \end{itemize}
    Then, a subset $S$ is independent if it $\mathrm{rk}(S) = |S|$.
\end{definition}

Standard references (c.f. Theorem 2.2.6 in \textit{Theory of Matroids} \cite{White86}) show that definitions in terms of independent sets, bases, and circuits, rank (Definitions \ref{def:matroid_independence},\ref{def:matroid_basis}, \ref{def:matroid_circuit}, and \ref{def:matroid_rank}) are equivalent.

\begin{definition}\label{def:matroid_strong_map}
    Let $\bullet$ be the rank $0$ matroid on $1$ element.
    Given matroids $M_1$, $M_2$ on finite sets $E_1$, $E_2$ respectively, a \textbf{pointed strong map} $f: M_1 \rightarrow M_2$
    is a set function $f: E_1 \sqcup \bullet \rightarrow E_2 \sqcup \bullet$ which sends $\bullet \mapsto \bullet$ satisfying either of the equivalent properties:
    \begin{itemize}
        \item the preimage of any flat in $M_2$ is a flat in $M_1$
        \item for all $A \subset B \subset E_1$, $\mathrm{rk}_{M_2}(f(B)) - \mathrm{rk}_{M_2}(f(A)) \leq \mathrm{rk}_{M_1}(B) - \mathrm{rk}_{M_1}(A)$
    \end{itemize}
    
    Matroids with pointed strong maps form a category which we denote by $\texttt{Matroid}_\bullet$.
\end{definition}

Given an existing matroid, it is possible to construct new matroids by restricting, deleting, or contracting a subset.
\begin{definition}\label{def:matroid_restriction_deletion_contraction}
    Let $M$ be a matroid on finite set $E$, with rank function $\mathrm{rk}$. Given a subset $S \subset E$, 
    \begin{itemize}
        \item the \textbf{restriction} of $M$ to $S$ is a matroid $M_{|S}$ on $S$ whose independent sets are the independent sets of $M$ contained in $S$. It has rank function $\mathrm{rk}_{M_{|S}}(A) = \mathrm{rk}(A)$. The inclusion $S \sqcup \bullet \hookrightarrow E \sqcup \bullet$ is a pointed strong map from the restriction $M_{|S} \rightarrow M$.
        
        \item the \textbf{deletion} of $M$ by $S$ is the restriction to the complement $E \setminus S$.
        
        \item the \textbf{contraction} of $M$ to $S$ is a matroid $M_{/S}$ on the complement $E \setminus S$ whose subsets are independent if their union with every independent subset of $S$ is independent in $M$. It has rank function $\mathrm{rk}_{M_{/S}}(A) = \mathrm{rk}(A \cup S) - \mathrm{rk}(S)$. The surjection $E \sqcup \bullet \hookrightarrow (E\setminus S) \sqcup \bullet$, which maps elements $i \mapsto \bullet$ if $i \in S$, is a pointed strong map to the contraction $M \rightarrow M_{/S}$.
    \end{itemize}
\end{definition}

Given two existing matroids, it is possible to construct a new matroid by combining them by direct sum.
\begin{definition}\label{def:matroid_direct_sum}
    Let $M_1$, $M_2$ be matroids on finite sets $E_1$, $E_2$ respectively. The \textbf{direct sum} $M_1 \oplus M_2$ is a matroid on the disjoint union $E_1 \sqcup E_2$, whose bases are the disjoint union of a basis in $E_1$ and a basis in $E_2$. Or equivalently the circuits are the disjoint union of a circuit in $E_1$ and a circuit in $E_2$. For $i = 1, 2$ The inclusions $E_i \sqcup \bullet \hookrightarrow E_1 \sqcup E_2 \sqcup \bullet$ are pointed strong maps to the direct sum $M_i \rightarrow M_1 \oplus M_2$.
\end{definition}
The direct sum is the coproduct in $\texttt{Matroid}_\bullet$.

Given an existing matroid, one may construct a dual matroid.
\begin{definition}\label{def:matroid_dual}
    Let $M$ be a matroid on finite set $E$. The \textbf{dual matroid} is a matroid $M^*$ on whose bases are complements of the bases on $M$. It has rank function 
    $\mathrm{rk}_{M^*}(A) = \mathrm{rk}_{M}(E \setminus A) + |A| - \mathrm{rk}_{M}(E)$.
\end{definition}

\begin{definition}\label{def:matroid_simple}
    Let $M$ be a matroid on finite set $E$. 
    \begin{itemize}
        \item An element $i \in E$ is called a \textbf{loop} if the singleton $\{i\}$ is circuit.
        
        \item Two elements $i,j \in E$ are called \textbf{parallel} if the tuple $\{i, j\}$ is circuit.
    \end{itemize}
    A matroid with no loops or parallel elements, or equivalently no dependent sets with size $< 3$, is called \textbf{simple}.
    
    Each matroid has unique simplification up to isomorphism by deleting all loops and identifying all elements of the same parallel class.

    We denote the full subcategory in $\texttt{Matroid}_\bullet$ of simple matroids by $\texttt{Matroid}^{\mathrm{simple}}_\bullet$.
    
    Given a surjective pointed set function $f: E_1 \sqcup \bullet \rightarrow E_2 \sqcup \bullet$, and a matroid $M$ on $E_2$, define the pull-back matroid $f^{-1}M$ on $E_1$ which has rank function for $A \subset E_1$,
    $$\mathrm{rk}_{f^{-1}M}(A) = \mathrm{rk}_{M}(f(A)).$$
    
    If $M$ is simple, then each fiber of $f$ is a set of loops or a parallel class. This is because for $i,j \in E_1$, if $f(i) = f(j)$, then $\mathrm{rk}_{f^{-1}M}(ij) = \mathrm{rk}_{f^{-1}M}(f(i)) \leq 1$ so $i$ and $j$ are parallel or both loops. So $M$ is the simplification of $f^{-1}M$ by identifying the fibers.
    
    We define a subset $S \subset E_2$ to be a \textbf{circuit up to simplification by $f$} if $f^{-1}(S)$ is a circuit in $f^{-1}M$.
\end{definition}

\begin{lemma}\label{lem:strong_map_circuits}
    When $M_1$ and $M_2$ are simple matroids, a pointed set function $f: E_1 \sqcup \bullet \rightarrow E_2 \sqcup \bullet$ is a pointed strong map if and only if the image of every circuit in $M_1$ is a union of circuits up to simplification by $f$ in $M_2$.
    \begin{proof}
        Every pointed set function $f: E_1 \sqcup \bullet \rightarrow E_2 \sqcup \bullet$ factors as a surjective map onto the image, followed by an embedding. 
        $$E_1 \sqcup \bullet \rightarrow \mathrm{im}(f) \sqcup \bullet \rightarrow E_2 \sqcup \bullet$$
        Since the restriction $M_2|_{\mathrm{im}(f)} \rightarrow M_2$ preserves all ranks, it is a strong map and preserves cycles. Hence, without loss of generality, it suffices to consider the case where $f$ is surjective.
        
        By pullback along $f$, we can further reduce to the case where $E_1 = E_2$ up to simplification.
        
        By Proposition 8.1.6.(g) in \cite{White86} pointed set functions $f: E_1 \sqcup \bullet \rightarrow E_2 \sqcup \bullet$ with $E_1 = E_2$ are strong maps if and only if the image of each circuit is a union of circuits.
    \end{proof}
\end{lemma}

Matroids generalize linear independence of vectors from linear algebra, from which the terms ``independence", ``dependence", ``basis", and ``rank" are borrowed.
\begin{example}\label{eg:realization}
    Let $k$ be a field, and $V$ be a $k$-vector space. Given a finite collection of vectors $\{v_i\}_{i \in E}$ with $v_i\in V$ (indexed by $E$), the subsets $S \subset E$ which index a linearly independent set of the vectors $\{v_i\}_{i \in S}$ form the independent sets of a matroid. i.e.,
    $$S \in \mathcal{I} :\iff \{v_i\}_{i \in S} \text{ is linearly independent in }V.$$
\end{example}
    
\begin{definition}\label{def:realizable}
    Let $k$ be a field. A matroid is \textbf{$k$-realizable} (or \textbf{$k$-representable}) if there exists a $k$-vector space $V$ and vector arrangement $\{v_i\}_{i \in E} \in V^E$, so that the independent sets of the matroid index the linearly independence of vectors in the arrangement (as described in Example \ref{eg:realization}).
\end{definition}

\begin{example}\label{eg:realizable_matroid}
    Over field $k = \mathbb{R}$, let $\{v_0 , \dots , v_3\} \subset \mathbb{R}^2$ be the four column vectors of the matrix 
    $$\begin{bmatrix}
        1 & 0 & 1 & 2\\
        0 & 1 & 1 & 2
    \end{bmatrix};$$
    this vector arrangement realizes a matroid with 
    \begin{itemize}
        \item independent sets: $0, 1, 2, 3, 01, 02, 03, 12, 13$
        \item bases: $01, 02, 03, 12, 13$
        \item circuits: $012, 013, 23$
        \item flats: $\emptyset, 0, 1, 23, 0123$
    \end{itemize}
\end{example}

There exist matroids which are not realizable over any field. One notable example is the V\'amos matroid on eight elements $E = \{0, \dots, 7\}$ whose bases are $$\mathcal{B}={E\choose 4} \setminus \{0123, 0145, 0167, 2345, 2367\},$$ i.e., all the four element subsets except the five listed circuits.

A natural question which arises: ``Are there algebraic structures generalizing fields and vector spaces in which linear independence is well-defined, and whose ``vector arrangements" realize all matroids?" Such algebraic structures generalizing fields are the topic of the next section.

\section{Algebraic Structures: Bands, Idylls, Hyperrings, Hyperfields}\label{sec:band_idyll}

\begin{definition}[Baker-Jin-Lorscheid, 2025 \cite{BakerJinLorscheid25}]
    A pointed monoid $B$ is a commutative multiplicative monoid with distinguished elements $0,1 \in B$, such that, for all $a \in B$, $0 \cdot a = 0$ and $1 \cdot a = a$.
    $$B^+ := \mathbb{N}[B \setminus \{0\}] = \{\text{formal sums of nonzero elements in $B$}\}$$
    By definition, the formal sums are taken to be commutative. The zero element $0 \in B$ is identified with the empty sum in $B^+$. The set of formal sums $B^+$ is a semigroup under concatenation of sums.
    
    A \textbf{band} consists of a pointed monoid $B$ and a \textbf{null set} $N_B \subset B^+$ which satisfies
    \begin{itemize}
        \item $0 \in N_B$
        \item $N_B + N_B = N_B$
        \item $B \cdot N_B = N_B$
        \item for all $a \in B$, there exists a unique ``additive inverse" $b \in B$ such that $a+b \in N_B$. Denote this by $-a$.
    \end{itemize}
    An \textbf{idyll} is a band for which all nonzero elements have multiplicative inverse.
\end{definition}
All commutative rings $R$ are bands with null set $$N_R = \left\{\sum a_i \in R^+ \mid \sum a_i = 0 \text{ in $R$}\right\}$$ and likewise all fields are idylls.

\begin{definition}[Baker-Jin-Lorscheid, 2025 \cite{BakerJinLorscheid25}]\label{def:band_idyll}
    Given bands $B, C$, a function $f:B \rightarrow C$ is a \textbf{band morphism} if 
    \begin{itemize}
        \item $f(0_B) = 0_C$
        \item $f(1_B) = 1_C$
        \item for all $a,b \in B$, $f(ab) = f(a)f(b)$
        \item for all $\sum a_i \in N_B$, $\sum f(a_i) \in N_C$
    \end{itemize}
    Bands with band morphisms form a category which we denote by $\texttt{Band}$. Idylls form a full subcategory $\texttt{Idyll}$.
\end{definition}

\begin{definition}\label{def:hyperring_hyperfield}
    Given a band $B$, define a (possibly empty or multivalued) binary operation $\boxplus: B \times B \rightarrow 2^B$
    $$x \boxplus y := \{z \in B \mid -z + x + y  \in N_B\}$$
    which extends to subsets $X, Y \subset B$
        $$X \boxplus Y := \bigcup_{x \in X, y \in Y} x \boxplus y$$
    A band $B$ is additionally a \textbf{commutative hyperring} if the operation $\boxplus$ satisfies
    
    \begin{itemize}
        \item for all $x,y \in B$, $x \boxplus y \neq \emptyset$ 
        \item (associativity) for all $x,y,z \in B$, $(x \boxplus y)\boxplus z = x \boxplus (y \boxplus z)$
    \end{itemize}
    A commutative hyperring is a \textbf{hyperfield} if all nonzero elements have a multiplicative inverse.

    We denote their full subcategory in $\texttt{Band}$ by $\texttt{CommHyperring}$ and $\texttt{Hyperfield}$.
\end{definition}
All commutative rings are commutative hyperrings with singleton hyperoperation $a \boxplus b = \{a+b\}$. Likewise all fields are hyperfields. Therefore, we have the following inclusions:

\[\begin{tikzcd}
	{\texttt{Field}} & {\texttt{Hyperfield}} & {\texttt{Idyll}} \\
	{\texttt{CommRing}} & {\texttt{CommHyperring}} & {\texttt{Band}}
	\arrow[hook, from=1-1, to=1-2]
	\arrow[hook, from=1-1, to=2-1]
	\arrow[hook, from=1-2, to=1-3]
	\arrow[hook, from=1-2, to=2-2]
	\arrow[hook, from=1-3, to=2-3]
	\arrow[hook, from=2-1, to=2-2]
	\arrow[hook, from=2-2, to=2-3]
\end{tikzcd}\]

\begin{definition}
    Let $B$ be a band. A subset $I \subseteq B^+$ is a \textbf{null ideal} if it satisfies the following properties
    \begin{itemize}
        \item $0 \in I$
        \item $I + I = I$
        \item $B \cdot I = I$
        \item $N_B \subseteq I$
        \item (substitution rule) for all $a,b,c_i \in B$,\\
        if $a + (-b) \in I$ and $b + \sum c_i \in I$, then $a + \sum c_i \in I$. 
    \end{itemize}
\end{definition}
The null set $N_B$ is a null ideal. Given a subset $S \subseteq B^+$, write $\langle S\rangle$ for the null ideal generated by $S$.
\begin{definition}
    Let $B_1$ and $B_2$ be bands. Given a band morphism $f: B_1 \rightarrow B_2$, the \textbf{nullkernel} is the set
$$\mathrm{nullker}(f) := \left\{\sum a_i \in B_1^+ \mid \sum f(a_i) \in N_{B_2}\right\}$$
\end{definition}
The nullkernel of a band morphism $f: B_1 \rightarrow B_2$ is a null ideal $\subseteq B_1^+$.

\begin{definition}\label{def:nullidealquotient}
    Given a band $B$ and null ideal $I$, define the \textbf{quotient band $B / I := B / \sim$} under the equivalence relation $$(\text{for all } a,b\in B:) \quad a \sim b : \iff a+(-b)\in I$$ with the null set 
    $$N_{B / I} := \left\{\sum[a_i] \mid \text{for }a_i \in B, \text{such that} \sum a_i \in I\right\}$$

    The substitution rule in $I$ assures uniqueness of additive inverses in $N_{B/I}$.
\end{definition}
The quotient map $B \rightarrow B / I$ sending each element to its equivalence class $a \mapsto [a]$ has nullkernel $I$. Given a surjective band morphism $f: B_1 \rightarrow B_2$, there is an isomorphism $B_2 \cong B_1/\mathrm{nullker}(f)$ (Corollary 1.14 in \cite{BakerJinLorscheid25}).

\begin{example}
    The \textbf{Krasner hyperfield} is an idyll $\mathbb{K} = \{0, 1\}$ with null set $$N_{\mathbb{K}} = \{0, 1+1,1+1+1, \cdots\}.$$ 
    It is a hyperfield with hyperoperation
    $$\begin{tabular}{c|c c}
        $\boxplus$ & $0$ & $1$ \\
        \hline
        $0$ & $0$ & $1$\\
        $1$ & $1$ & $\{0, 1\}$
    \end{tabular}$$
    
\end{example}
 The Krasner hyperfield encodes the arithmetic of being zero or nonzero. It is the terminal object in $\texttt{Idyll}$, since for any idyll, there is a unique morphism to $\mathbb{K}$ which sends zero to $0$ and anything nonzero to $1$.

\begin{example}        
    The \textbf{sign hyperfield} is an idyll $\mathbb{S} = \{0, 1, -1\}$
    with null set $$N_{\mathbb{S}} = \{0\} \cup \left\{\sum a_i \in \mathbb{S}^+\mid \text{both $1$ and $-1$ appear among $a_i$'s} \right\}.$$
        
    It is a hyperfield with hyperoperation
    $$\begin{tabular}{c|c c c}
        $\boxplus$ & $0$ & $1$ & $-1$ \\
        \hline
         $0$ & $0$ & $1$ & $-1$\\
         $1$ & $1$ & $1$ & $\{0,1,-1\}$\\
         $-1$ & $-1$ & $\{0,1,-1\}$ & $-1$\\
    \end{tabular}$$
\end{example}
The sign hyperfield encodes the arithmetic of being zero, positive, or negative.

\begin{example}
    The \textbf{tropical hyperfield} is an idyll $\mathbb{T} = \{-\infty\} \cup \mathbb{R}$ with multiplication $x \cdot_{\mathbb{T}} y := x + y$, $0_\mathbb{T} = -\infty$, $1_\mathbb{T} = 0$, and has null set $$N_{\mathbb{T}} = \left\{\sum a_i \in \mathbb{\mathbb{T}}^+\mid \text{the maximum among $a_i$'s appears at least twice}\right\}$$
    It is a hyperfield with hyperoperation
        $$x \boxplus y = 
        \begin{cases}
            \max(x,y) &\text{if $x \neq y$}\\
            [-\infty,x] &\text{if $x = y$}
        \end{cases}$$
\end{example}
The tropical hyperfield encodes the arithmetic of a real valuation on a ring (c.f. the degree of a Puiseux series).

\section{Modules over Bands, Idylls, Hyperrings, Hyperfields}\label{sec:band_module}
In \cite{JarraLorscheidVital26}, Jarra, Lorscheid, Vital introduced modules over a band, analogous to modules over a commutative ring. Despite the name, they differ from ordinary modules in the sense that the underlying structure is not an abelian group, but rather a pointed set with a null set of formal sums.

\begin{definition}
    Let $B$ be a band. Let $V$ be a pointed set with distinguished element $0_V \in V$.
    $$V^+ = \mathbb{N}[V \setminus 0_V] = \{\text{formal sums of nonzero elements of }V\}$$
    As with bands, formal sums are commutative, and $V^+$ is a semigroup and $0_V$ is identified with the empty sum.

    A \textbf{module over $B$}, or \textbf{$B$-module} is a pointed set $V$ equipped with a scalar multiplication $\cdot: B \times V \rightarrow V$ with a null set $N_V \subset V^+$ satisfying
    \begin{itemize}
        \item for $a,b \in B$, $v\in V$: $(a b)\cdot v = a \cdot(b\cdot v)$
        \item for $v \in V$, $1 \cdot v = v$, $0_B \cdot v = 0_V$, and for $a \in B$, $a \cdot 0_V = 0_V$.
        
        \item $0_V \in N_V$
        \item for all $v \in V$, there exists a unique ``additive inverse" $-v \in V$ such that $v + (-v) \in N_V$.
    \end{itemize}
   Scalar multiplication lifts to $B^+ \times V^+ \rightarrow V^+$, by requiring it to be distributive on formal sums: $(a+b)\cdot v := a\cdot v + b \cdot v$ and $a\cdot (v + w) := a\cdot v + a \cdot w$. It must furthermore satisfy properties
    \begin{itemize}
        \item $N_B \cdot V^+ \subset N_V$
        \item $B^+ \cdot N_V \subset N_V$
    \end{itemize}
    Since $v + (-1)\cdot v = (1+(-1))\cdot v \in N_V$, it follows that $(-1)\cdot v = -v$ by uniqueness of opposites.

    Let $V_1, V_2$ be $B$-modules. A \textbf{morphism of $B$-modules} or a $B$-\textbf{linear map} is a map $f: V_1 \rightarrow V_2$ such that
    \begin{itemize}
        \item $f(0) = 0$
        \item for all $a \in k, v \in V_1$, $f(a \cdot v) = a \cdot f(v)$
        \item for all $a_i \in V_1$, 
        if $\sum a_i \in N_{V_1}$ then $\sum f(a_i) \in N_{V_2}$.
    \end{itemize}
    By the third property, $f$ induces a semigroup homomorphism $f^+: V_1^+ \rightarrow V_2^+$, for which $f^+(N_{V_1}) \subseteq N_{V_2}$.

    Modules over a band $B$ and linear maps form a category, which we denote by $\texttt{Module}_B$.
\end{definition}
    
\begin{definition}
    Let $B$ be a band, and let $V$ be a $B$-module.
    A subset $I \subset V^+$ is a \textbf{null ideal} by definition if it satisfies the following properties
    \begin{itemize}
        \item $0 \in I$
        \item $I + I = I$
        \item $B\cdot I = I$
        \item $N_V \subseteq I$
        \item (substitution rule) for all $u,w,v_i \in V$,\\
        if $u + (-w) \in I$ and $w + \sum v_i \in I$, then $u + \sum v_i \in I$. 
    \end{itemize}

    The null set $N_V$ is a null ideal. Given a subset $S \subseteq V^+$, write $\langle S\rangle$ for the null ideal generated by $V$.

    The quotient construction is similar to the quotient in bands. Given a $B$-module $V$ and null ideal $I$, define the \textbf{quotient module $V / I := V / \sim$} under the equivalence relation $$(\text{for all } a,b\in G:) \quad a \sim b : \iff a+(-b)\in I$$ with the null set 
    $$N_{V / I} := \left\{\sum[v_i] \mid \text{for }v_i \in V, \text{such that} \sum v_i \in I\right\}$$
    Since $I$ is stable under the action of $B$, scalar multiplication is well defined on equivalence classes.
\end{definition}
The map sending each element to its equivalence class $V \rightarrow V / I$ is a linear map. Also, given a linear map $f:V\rightarrow W$, the nullkernel $$\mathrm{nullker}(f)= \left\{\sum v_i \in V^+ \mid \sum f(v_i) \in N_W\right\}$$ is a null ideal $\subset V^+$, and there is an isomorphism $W \cong V/\mathrm{nullker}(f)$.

\begin{definition}\label{def:submodule}
    Let $B$ be a band, and let $V$ be a $B$-module. Let $W \subset V$ be a subset which is stable under the scalar multiplication ($B\cdot W = W$). Then, $W$ is also a $B$-module, with null set 
    $N_W = \left\{ \sum v_i \in W^+ \mid \sum v_i \in N_V\right\}$
    The $B$-module $W$ is called a \textbf{strict submodule} of $V$. The inclusion map $W \hookrightarrow V$ is a linear map by construction.
\end{definition}

\begin{definition}\label{def:quotient_by_submodule}
    Let $B$ be a band, and let $V$ be a $B$-module. Let $W \subset V$ be a strict submodule. Then the quotient 
    $V / \langle W \rangle$ is $B$-module. The projection map $V \rightarrow V/\langle W\rangle$ is a linear map.
\end{definition}
Not all null ideals are generated by a strict submodule, and so not all quotient maps are quotients by a strict submodule. Thus, we make the distinction between quotients by a null ideal and by a strict submodule.

In \cite{Hamada26}, Hamada investigated basic categorical properties of $\texttt{Module}_B$. In particular, it has all products, coproducts, equalizers, and coequalizers, and hence is complete and cocomplete (Proposition 4.17 in \cite{Hamada26}).

\begin{example}
    Let $B$ be a band. 
    \begin{itemize}
        \item The single point $\{0\}$ is a trivial $B$-module, with null set $\{0\}$. It is the zero object in $\texttt{Module}_B$.
        \item The band $B$ itself is a $B$-module with scalar multiplication being the usual multiplication.
    \end{itemize}
\end{example}

\begin{definition}
    Let $B$ be a band, and $\{V_i\}_{i \in I}$ be a family of $B$-modules, indexed by $I$. The \textbf{product} is a $B$-module $\prod_{i \in I}V_i$. Its underlying pointed set is the usual Cartesian product with zero $0 := (0_{V_i})_{i \in I}$. Its null set is 
    $$N_{\prod_{i \in I}V_i} := \left\{\sum_j (v_{ij})_{i\in I} \mid \text{for all $i \in I$, } \sum_j v_{ij} \in N_{V_i}\right\}$$
    and scalar multiplication defined component-wise $a\cdot (v_{i})_{i\in I} := (a \cdot v_{i})_{i\in I}$.
\end{definition}
The product $\prod_{i \in I}V_i$ satisfies the universal property of the product in $\texttt{Module}_B$.
\begin{definition}
    Let $B$ be a band, and $\{V_i\}_{i \in I}$ be a family of $B$-modules, indexed by $I$. The \textbf{coproduct} or \textbf{direct sum} is a $B$-module $\bigoplus_{i \in I}V_i$. Its underlying pointed set is the quotient of the disjoint union by identifying zeros
    $\bigoplus_{i \in I}V_i := (\bigsqcup_{i \in I} V_i) / \{\text{for all $i \in I$, } 0 \sim 0_{V_i}\}$ (this is the coproduct as pointed sets). Its null set is 
    $$N_{\bigoplus_{i \in I}V_i} := \langle N_{V_i} \mid i\in I\rangle$$
    and scalar multiplication defined by extending the multiplication on each component $V_i$.
\end{definition}
The direct sum $\bigoplus_{i \in I}V_i$ satisfies the universal property of the coproduct in $\texttt{Module}_B$.
\begin{definition}
    Let $B$ be a band and $f,g:V \rightarrow W$ be linear maps of $B$-modules $V$ and $W$.
    The \textbf{equalizer} is the strict submodule
    $$\mathrm{eq}(f,g) := \{v \in V \mid f(v) = g(v)\} \subset V.$$
    Consequently, the \textbf{kernel} of a linear map $f:V \rightarrow W$ can be constructed as the equalizer of $f$ with the zero map.
    $$\mathrm{ker}(f) := \{v \in V \mid f(v) = 0\}$$
\end{definition}
The equalizer $\mathrm{eq}(f,g)$ satisfies the universal property of the equalizer in $\texttt{Module}_B$. The kernel $\mathrm{ker}(f)$
satisfies the universal property of the kernel in $\texttt{Module}_B$.

\begin{definition}
    Let $B$ be a band and $f,g:V \rightarrow W$ be linear maps of $B$-modules $V$ and $W$.
    The \textbf{coequalizer} is the quotient of $W$ by the null ideal generated by formal differences $f(v) - g(v)$
    $$\mathrm{coeq}(f,g) := W/\langle f(v) - g(v) \mid v\in V \rangle.$$
    Consequently, the \textbf{cokernel} of a linear map $f:V \rightarrow W$ can be constructed as the coequalizer of $f$ with the zero map, which is the same as the quotient of $W$ by the strict submodule image of $f$.
    $$\mathrm{coker}(f) := W/\langle f(v) \mid v\in V \rangle$$
\end{definition}
The coequalizer $\mathrm{coeq}(f,g)$ satisfies the universal property of the coequalizer in $\texttt{Module}_B$. The cokernel $\mathrm{coker}(f)$ satisfies the universal property of the cokernel in $\texttt{Module}_B$.

\begin{definition}
    Let $B$ be a band and let $V$ and $W$ be $B$-modules. The set of linear maps $V \rightarrow W$ is denoted $\mathrm{Hom}(V, W)$.
    The set $\mathrm{Hom}(V, W)$ has the structure of $B$-module when equipped with the null set
    $$N_{\mathrm{Hom}(V, W)} = \left\{\sum f_i \in \mathrm{Hom}(V, W)^+ \mid \text{$\sum f_i(v) \in N_W$ for all $v \in V$}\right\}.$$
    where scalar multiplication by $a \in B$ is defined by 
    $(a \cdot f)(v) := a\cdot f(v)$.
\end{definition}

\section{Matroids over Idylls and Hyperfields}\label{sec:valuated_matroid}

We survey existing constructions for matroids over an idyll $k$, or $k$-matroids. There are distinct notions of strong and weak $k$-matroids. We only consider strong $k$-matroids in this paper. Additionally, we later restrict to the class of perfect idylls $k$, for which strong and weak $k$-matroids coincide.

\begin{definition}[Baker-Lorscheid, 2021 \cite{BakerLorscheid21}]
    \noindent Let $k$ be an idyll. A (strong) \textbf{$k$-valued Pl\"ucker coordinate of rank $r$} is a point $p = (p_S)_{S \in {E \choose r}} \in  k^{E \choose r}$ satisfying the Pl\"ucker relations:
    for $I, J\subset E$, $|I| = r+1, |J| = r-1$, $|I\setminus J| \geq 3$
     
    $$\sum_{i \in I\setminus J} (-1)^{\epsilon(i;I,J)} p_{I\setminus i} p_{J \cup i} \in N_k$$
    where $\epsilon(i;I,J) = |\{j \in I \cup J \mid j < i \}|$, relative to a fixed order on $E$.
\end{definition}

These generalize the usual Pl\"ucker relations for the embedding of the Grassmanian $\mathrm{Gr}_{k}(n,r)$ into  $\mathbb{P}(k^{E \choose r})$, when $k$ is a field.

The support $\mathrm{supp}(p)=\{B \in {E \choose r} \mid p_B \neq 0\}$ is the set of bases for a matroid of rank $r$, called the underlying matroid.

Matroids over $k$ were first defined in terms of Pl\"ucker coordinates. However, in the paper we will exclusively consider an alternate axiomatization in terms of vector sets:

\begin{definition}[Anderson, 2019 \cite{Anderson19}]\label{def:vector_set}
    Let $\mathcal{V} \subset k^E$.
    \begin{itemize}
        \item $B \subset E$ is a \textbf{support basis} for $\mathcal{V}$ if for all nonzero $v \in \mathcal{V}$, $\mathrm{supp}(v) \cap B \neq \emptyset$ and $B$ is minimal with this property.
        \item $v^1, \dots, v^r \subset \mathcal{V}$ is a \textbf{reduced row echelon form (RREF)} for $B$, if the $r \times B$ minor of the matrix whose rows are $v^1, \dots, v^r$ is the identity.
    \end{itemize}
    For $v^1, \dots, v^r \in k^E$ define the span
    \begin{align*}
        \mathrm{span}(v^1, \dots, v^r) = \{ &w \in k^E \mid  \text{there exists } a^1, \dots, a^r \in k,\\ &\text{ for } i \in E, -w_i + a^1v_i^1 + \cdots + a^nv_i^r \in N_k \}
    \end{align*}
    $\mathcal{V}$ is a $k$-\textbf{vector set} if each support basis has a RREF, and
    $$\mathcal{V} = \bigcap_{\text{$B$ is a support basis }}\bigcap_{\text{$v^1, \dots, v^r$ is a RREF for $B$}} \mathrm{span}(v^1, \dots, v^r)$$
\end{definition}
Vector sets generalize the row space of an $|E| \times r$ matrix when $k$ is a field. Over a field, the row space is preserved under elementary row operations, and hence is the span of RREFs in any support basis. Over an idyll, this is not necessarily the case, hence the need to take the intersection over all support bases and RREFs.

Anderson also showed the following properties of $k$-vector sets:
\begin{proposition}[Anderson, 2019 \cite{Anderson19}]
\:\
\begin{itemize}
    \item The RREFs of a $k$-vector set are unique for each support basis (Proposition 2.11 in \cite{Anderson19}).
    \item The support bases of a $k$-vector set satisfy the basis exchange property (Lemma 2.13 in \cite{Anderson19}), and are dual to the bases of the underlying matroid.
    \item There is a one-to-one correspondence between $k$-vector sets and (strong) $k$-Pl\"ucker coordinates (and other axiomatizations of (strong) $k$-matroids) (Theorem 2.18 in \cite{Anderson19}).
\end{itemize}
\end{proposition}
The support of the RREFs precisely indicate the circuits of the underlying matroid.

\begin{example}
    $\mathcal{V}=\{0000, 1110, 1101, 0011, 1111\} \subset \mathbb{K}^4$ is a $\mathbb{K}$-vector set. Indexing the coordinates by $\{0, \dots, 3\}$, its support bases are 
    \begin{itemize}
        \item $02$, with RREF $\{1101, 0011\}$
        \item $03$, with RREF $\{1110, 0011\}$
        \item $12$, with RREF $\{1101, 0011\}$
        \item $13$, with RREF $\{1110, 0011\}$
        \item $23$, with RREF $\{1110,1101\}$
    \end{itemize}
    The support bases are complementary to the bases of the matroid in Example \ref{eg:realizable_matroid}.
    $\mathcal{V}$ is a $\mathbb{K}$-vector set because the spans of RREFs are 
    \begin{itemize}
        \item $\mathrm{span}(1101, 0011) = \{0000, 1101, 0011, 1110, 1111\} = \mathcal{V}$
        \item $\mathrm{span}(1110, 0011) = \{0000, 1101, 0011, 1101, 1111\} = \mathcal{V}$
        \item $\mathrm{span}(1110, 1101) = \{0000, 1110, 1101, 0011,1011,0111, 1111\} = \mathcal{V} \cup \{1011,0111\}$
    \end{itemize}
    and their intersection is precisely $\mathcal{V}$. The RREFs are minimally supported among nonzero elements of $\mathcal{V}$, supported on $012, 013, 23$, which are the circuits of the matroid in Example \ref{eg:realizable_matroid}.
\end{example}

\begin{definition}
    Let $k$ be an idyll. Two elements $a=(a_i)_{i \in E}$, $b=(b_i)_{i \in E} \in k^E$ are orthogonal, denoted $a \perp b$, if 
    $$\sum_{i \in E} a_ib_i \in N_k.$$
    Two subsets $X, Y \subset k^E$ are orthogonal, denoted $X \perp Y$, if $a \perp b$ for all $a \in X$ and all $b\in Y$.
    Given a subset $S \subset k^E$, define the \textbf{orthogonal complement}
    $$S^\perp := \{b \in k^E \mid a\perp b \text{ for all $a \in S$}\}.$$
\end{definition}

\begin{definition}
    Let $k$ be an idyll. The \textbf{support} of an element $a=(a_i)_{i \in E} \in k^E$ is the subset of $E$ for which the component of $a$ is nonzero.
    $$\mathrm{supp}(a) := \{i \in E \mid a_i \neq 0\}$$
    Given a subset $S \subset k^E$, define the minimal support
    $$\mathrm{minsupp}(S) := \{a \in S \mid \text{$\mathrm{supp}(a)$ is minimal with respect to inclusion among elements of $S$}\}.$$
\end{definition}

\begin{definition}
    Let $k$ be an idyll, and let $\mathcal{V} \subset k^E$ be a $k$-vector set.
    The \textbf{dual vector set} $\mathcal{V}^*$ is the orthogonal complement of the minimally supported elements of $\mathcal{V}$
    $$\mathcal{V}^* = \mathrm{minsupp}(\mathcal{V} \setminus \{0\})^\perp$$
    $\mathcal{V}^*$ is also a $k$-vector set (Theorem 2.18 in \cite{Anderson19}).
\end{definition}
    The dual vector set generalizes duality of matroids in the sense that if $M$ is the underlying matroid of $\mathcal{V}$, then $M^*$ is the underlying matroid of $\mathcal{V}^*$.

    Since $\mathrm{minsupp}(\mathcal{V} \setminus \{0\}) \subset \mathcal{V}$, the orthogonal complement of $V$ is contained in its dual, $\mathcal{V}^\perp \subset \mathcal{V}^*$. Unfortunately, this may be a proper inclusion in general and $\mathcal{V}^\perp$ may not be a $k$-vector set (a counterexample is given in Section 5.4.4 in \cite{Anderson19}).
\begin{definition}
    An idyll $k$ is called \textbf{perfect} if, for all $k$-vector sets $\mathcal{V}$, the above inclusion is equality; $\mathcal{V}^\perp = \mathcal{V}^*$.
\end{definition}

The perfect property will be desirable in the remainder of the paper. All doubly-distributive hyperfields (for which $(a \boxplus b)(c \boxplus d) = ac \boxplus ad \boxplus bc \boxplus bd$) are perfect idylls (Corollary 3.45 in \cite{BakerBowler19}). This includes all fields, $\mathbb{K}$, $\mathbb{S}$, and $\mathbb{T}$.

A notion of morphism for matroids over a perfect idyll $k$, in terms of their $k$-vector sets, was introduced in \cite{JarraLorscheidVital26} via multiplication by a submonomial matrix:
\begin{definition}
    Let $k$ be a perfect idyll. A matrix with entries in $k$ is \textbf{submonomial} if each row and column has at most one nonzero entry. A $E_2 \times E_1$ submonomial matrix (with rows indexed by $E_2$ and columns by $E_1$) defines a map $k^{E_1} \rightarrow k^{E_2}$ by matrix multiplication.

    For any $k$-vector sets $\mathcal{V}_1 \subset k^{E_1}$ and $\mathcal{V}_2 \subset k^{E_2}$, a \textbf{$k$-vector set morphism} is a map $f:k^{E_1} \rightarrow k^{E_2}$ given by multiplication by a submonomial matrix such that $f(\mathcal{V}_1) \subset \mathcal{V}_2$.
\end{definition}
Vector sets over a perfect idyll $k$ and their morphisms form a category, which we denote by $\texttt{Matoid}_k$.

Every submonomial map $f:k^{E_1} \rightarrow k^{E_2}$ has an associated map of pointed sets $\underline{f}: E_1 \sqcup \bullet \rightarrow E_2 \sqcup \bullet$, given by 
$$i \mapsto \begin{cases}
    j & \text{if the $ji$ entry is nonzero}\\
    \bullet & \text{if the $i$-th column is zero}
\end{cases}$$
If $f$ is a morphism of $k$-vector sets, then $\underline{f}$ is a pointed strong map of the underlying matroids (Proposition C in \cite{JarraLorscheidVital26}). This association is functorial, yielding a functor $\texttt{Matoid}_k \rightarrow \texttt{Matroid}_\bullet$. In the case for the Krasner hyperfield $k = \mathbb{K}$, the functor is a faithful embedding, essentially injective and surjective on objects. In this sense, $\mathbb{K}$-matroids are the same as ordinary matroids. However, the functor is not full; i.e., not all pointed strong maps arise as the image of a morphism of $\mathbb{K}$-vector sets. Due to the fact that submonomial matrices have at most one nonzero entry per column, only the pointed strong maps $f: E_1 \sqcup \bullet \rightarrow E_2 \sqcup \bullet$ arise for which, for each $j \in E_2$, there exists at most one $i\in E_1$ such that $f(i) = j$.

There are analogous constructions of restriction, deletion, contraction, direct sum for $k$-vector sets:
\begin{definition}
    Let $k$ be a perfect idyll and $\mathcal{V} \subset k^E$ be a $k$-vector set on a finite set $E$. Given a subset $S \subset E$,
    \begin{itemize}
        \item The \textbf{restriction} is the $k$-vector set $\subset k^S$ defined by $$\mathcal{V}_{|S} = \left\{(a_i)_{i \in S} \in k^S \mid \text{there exists $(b_i)_{i \in E} \in \mathcal{V}$ with $b_i = \begin{cases}a_i & i \in S\\ 0 & \text{otherwise}\end{cases}$ }\right\}.$$
        The restriction is equipped with a $k$-vector set morphism
        $\mathcal{V}_{|S} \rightarrow \mathcal{V}$ given by the $E \times S$ submonomial matrix whose $S \times S$ minor is the identity matrix, and zero on the other rows indexed by $E \setminus S$.
        
        \item The \textbf{deletion} is the restriction to the complement $E \setminus S$.
        
        \item The \textbf{contraction} is the $k$-vector set $\subset k^{E\setminus S}$ defined by $$\mathcal{V}_{/S} = \left\{(a_i)_{i \in S} \in k^{E \setminus S} \mid \text{there exists $(b_i)_{i \in E} \in \mathcal{V}$ with $b_i = a_i$ for all $i\in E \setminus S$}\right\}.$$
        The contraction is equipped with a $k$-vector set morphism
        $\mathcal{V} \rightarrow \mathcal{V}_{|S}$ given by the $(E \setminus S) \times E$ submonomial matrix whose $(E \setminus S) \times (E \setminus S)$ minor is the identity matrix, and zero on the other columns indexed by $S$.
    \end{itemize}
\end{definition}
It is necessary for the idyll $k$ to be perfect in order for the sets $\mathcal{V}_{|S}$ and $\mathcal{V}_{/S}$ to be consistent with the restriction and contraction of $k$-Pl\"ucker coordinates defined in \cite{BakerBowler19}. See Section 4.2. and Theorem in \cite{Anderson19} for a discussion of this issue.

\begin{definition}\label{def:vector_set_direct_sum}
    Let $k$ be a perfect idyll, and $\mathcal{V}_1 \subset k^{E_1}$ and $\mathcal{V}_2 \subset k^{E_2}$ be $k$-vector sets on finite sets $E_1, E_2$ respectively. The \textbf{direct sum} is the $k$-vector set $\subset k^{E_1 \sqcup E_2}$ defined by 
    \begin{align*}
        &\mathcal{V}_1 \oplus \mathcal{V}_2 = \\&\left\{(a_i)_{i \in E_1 \sqcup E_2} \in k^{E_1 \sqcup E_2} \mid \text{there exists $(b_i)_{i \in E_1} \in \mathcal{V}_1$, $(c_i)_{i \in E_2} \in \mathcal{V}_2$ 
        with $a_i = \begin{cases}b_i & i \in E_1\\ c_i & i \in E_2\end{cases}$} \right\}.
    \end{align*}
    For $i = 1,2$, there are $k$-vector set morphisms $\mathcal{V}_i \rightarrow \mathcal{V}_1 \oplus \mathcal{V}_2$ by restricting to $E_i$.
\end{definition}

\begin{definition}
    Let $k$ be a perfect idyll, and $E$ a finite set. Given a $k$-vector set $\mathcal{V} \subset k^E$,
    \begin{itemize}
        \item an element $i \in E$ is called a \textbf{loop} if it a loop in the underlying matroid.
        \item two elements $i,j \in E$ are called \textbf{parallel} if they are parallel in the underlying matroid.
    \end{itemize}
     A $k$-vector set with no loops or parallel elements is called \textbf{simple}. Each $k$-vector set has a unique simplification up to isomorphism by deleting all loops and all but one element of each parallel class.

     We denote the full subcategory in $\texttt{Matoid}_k$ of simple $k$-vector sets by $\texttt{Matroid}^{\mathrm{simple}}_k$.
\end{definition}

\begin{lemma}\label{lem:vector_set_scaling}
    Let $k$ be a perfect idyll and $\mathcal{V} \subset k^E$ a $k$-vector set. Denote by $\lambda(\mathcal{V})$ the set obtained from $\mathcal{V}$ by (non-isotropically) scaling each coordinate direction $i \in E$ by a nonzero $\lambda_i \in k$. i.e.,
    $$\lambda(\mathcal{V}) = \left\{(\lambda_ia_i)_{i \in E} \in k^E \mid (a_i)_{i \in E} \in \mathcal{V}\right\}$$
    Denote the element obtained by scaling $(a_i)_{i \in E}$ by $$\lambda (a_i)_{i \in E} := (\lambda_ia_i)_{i \in E}.$$
    Then $\lambda(\mathcal{V})$ is $k$-vector set isomorphic to $\mathcal{V}$.
    \begin{proof}
        The map $a \mapsto \lambda a$ is defined multiplication by the diagonal $E \times E$ matrix with entries $\lambda_i$, which is a submonomial matrix. Its inverse is the diagonal matrix with entries $\lambda_i^{-1}$. 
        The image of $\mathcal{V}$ under this map is $\lambda(\mathcal{V})$ by construction.
        Hence this map determines an isomorphism $\mathcal{V} \cong \lambda(\mathcal{V})$ and $\lambda(\mathcal{V})$ is a $k$-vector set.
    \end{proof}
\end{lemma}

\begin{lemma}\label{lem:vector_set_permutation}
    Let $k$ be a perfect idyll and $\mathcal{V} \subset k^E$ a $k$-vector set. Let $\sigma: E \rightarrow E$ be a permutation of the coordinate directions. Denote by 
    $$\sigma(\mathcal{V}) = \left\{(a_{\sigma (i)})_{i \in E} \in k^E \mid (a_i)_{i \in E} \in \mathcal{V}\right\}$$
    Denote the element obtained by permutation by $$\sigma(a_i)_{i \in E} := (a_{\sigma (i)})_{i \in E}.$$
    Then $\sigma(\mathcal{V})$ is $k$-vector set isomorphic to $\mathcal{V}$.
    \begin{proof}
        The map $a \mapsto \sigma(a)$ is defined multiplication by the $E \times E$ permutation matrix representing $\sigma$, which is a submonomial matrix. Its inverse is the transpose matrix representing the inverse permutation.
        The image of $\mathcal{V}$ under this map is $\sigma(\mathcal{V})$ by construction.
        Hence this map determines an isomorphism $\mathcal{V} \cong \sigma(\mathcal{V})$ and $\sigma(\mathcal{V})$ is a $k$-vector set.
    \end{proof}
\end{lemma}

\begin{definition}\label{def:vector_set_symmetric}
    Let $k$ be a perfect idyll and $\mathcal{V} \subset k^E$ a $k$-vector set. Let $\sigma: E \rightarrow E$ be a permutation. We call $\mathcal{V}$ \textbf{symmetric with respect to $\sigma$} if $\mathcal{V} = \sigma(\mathcal{V})$.
\end{definition}

\begin{definition}\label{def:vector_set_duplicate}
    Let $k$ be a perfect idyll and $\mathcal{V} \subset k^E$ a $k$-vector set. We denote $e_i$ to be the standard coordinate basis vectors, so that $\sum_{i \in E} a_i e_i := (a_i)_{i \in E}$. We call two parallel elements $j, j' \in E$ \textbf{duplicate} in $\mathcal{V}$ if circuit $e_{j'} - e_j \in\mathcal{V}$.
\end{definition}

\begin{lemma}
	Let $k$ be a perfect idyll and let $\mathcal{V} \subset k^E$ be a $k$-vector set. If $j,j' \in E$ are duplicate in $\mathcal{V}$, then $\mathcal{V}$ is symmetric with respect to the transposition of $j$ and $j'$.
    \begin{proof}
        Let $\sigma: E \rightarrow E$ be the transposition of $j$ and $j'$. For any $\sum_{i \in E} x_i e_i \in \mathcal{V}$, we show that the transposed element $\sum_{i \in E} x_{\sigma(i)} e_i$ is also in $\mathcal{V}$.
        
        Since $e_j - e_j' \in \mathcal{V}$, every support basis $B \subset E$ contains $j$ or $j'$ or both.
        \begin{itemize}
            \item Suppose $B$ contains one but not the other (without loss of generality, suppose it contains $j$ but not $j'$). 

            For $\ell \in B$, let $\sum_{i \in E} R^\ell_i e_i$ denote the RREF. Since RREFs for $k$-vector sets are unique, the RREF for $B$ must contain $e_j - e_j'$. So we have $$R^j_i = \begin{cases}
                1 & i = j\\
                -1 & i = j'\\
                0 & \text{otherwise}
            \end{cases}$$
            Since $\sum_{i \in E} x_i e_i = x_j e_j + x_{j'}e_{j'} + \sum_{i \in E\setminus\{j,j'\}} x_i e_i$ is in the span of the RREF, we must have, for all $i \in E$,
            $$-x_i + \sum_{\ell \in B} a_\ell R^\ell_i \in N_k.$$
            As a result,
            $x_j = a_j$
            and
            \begin{align*}
            & -x_j -x_{j'}  + \sum_{\ell \in B \setminus j} a_\ell R^\ell_{j'}\\
            =& -x_{j'} -x_j + \sum_{\ell \in B \setminus j} a_\ell R^\ell_{j'} \\
            =& -x_{j'} + a_j R^j_{j'} + \sum_{\ell \in B \setminus j} a_\ell R^\ell_{j'} \in N_k.
            \end{align*}
            Then, the transposed element $x_{j'} e_j + x_je_{j'} + \sum_{i \in E\setminus\{j,j'\}} x_i e_i=\sum_{i \in E} x_{\sigma(i)} e_i$ is also in the span of the RREF.

            \item Suppose $B$ contains both.
            For $\ell \in B$, let $\sum_{i \in E} R^\ell_i e_i$ denote the RREF. Since $e_j - e_j'$ is in the span of the RREF, we must have
            \begin{align*}
                -1 + \sum_{\ell \in B} a_\ell R^\ell_j &\in N_k\\
                1 + \sum_{\ell \in B} a_\ell R^\ell_{j'} &\in N_k\\
                \sum_{\ell \in B} a_\ell R^\ell_{i} &\in N_k & \text{for $i \in E \setminus \{i,j\}$}
            \end{align*}
            These imply that $a_\ell = 0$ for $\ell \in B \setminus \{i,j\}$, $a_j = 1$, and $a_{j'} = -1$. As a result, we have $R^j_i = R^{j'}_i$ for $i \in E \setminus B$ so 
            $$R^{j}_i = R^{j'}_{\sigma(i)} \text{ for all $i\in E$}.$$ i.e., the $j$ and $j'$ rows of the RREF are symmetric with respect to the transposition. 
            Then, the transposed element $x_{j'} e_j + x_je_{j'} + \sum_{i \in E\setminus\{j,j'\}} x_i e_i=\sum_{i \in E} x_{\sigma(i)} e_i$ is also in the span of the RREF.
        \end{itemize}
        In any support basis, the transposed element $x_{j'} e_j + x_je_{j'} + \sum_{i \in E\setminus\{j,j'\}} x_i e_i=\sum_{i \in E} x_{\sigma(i)} e_i$ is in the span of the RREF. Therefore, $\sum_{i \in E} x_{\sigma(i)} e_i \in \mathcal{V}$, and $\mathcal{V}$ is symmetric with respect to the transposition $\sigma$.
    \end{proof}
\end{lemma}

\begin{definition}\label{def:vector_set_duplication}
    Let $\mathcal{V} \subset k^E$ and $\mathcal{V}' \subset k^{E \sqcup \{j'\}}$ be $k$-vector sets. We call $\mathcal{V}'$ a \textbf{duplication of $\mathcal{V}$ along $j \in E$} if 
    \begin{itemize}
        \item the elements $j$ and $j'$ are duplicate in $\mathcal{V}'$
        \item $\mathcal{V}$ is obtained from $\mathcal{V}'$ by deleting $j'$.
    \end{itemize}
\end{definition}

\begin{example}
    We give an example of duplication over the Krasner hyperfield $\mathbb{K}$. $\mathcal{V}=\{000, 111\} \subset \mathbb{K}^{\{0,1,2\}}$ is a $\mathbb{K}$-vector set. Its duplication in the last index $2$ is $\mathcal{V}' =\{0000, 1110, 1101,0011,1111\} \subset \mathbb{K}^{\{0,1,2,2'\}}$. Because $-1 = 1$ in $\mathbb{K}$, the last two indices $2$ and $2'$ are parallel, with circuit $e_{2'}-e_2=e_2+e_{2'}=0011 \in \mathcal{V}'$.
\end{example}

\begin{lemma}\label{lem:vector_set_duplication_unique}
    Let $k$ be a perfect idyll and $\mathcal{V} \subset k^E$ a $k$-vector set. The duplication $\mathcal{V}$ along $j \in E$ is unique.
    \begin{proof}
        Let $\mathcal{V}' \subset k^{E \sqcup \{j'\}}$ be a duplication along $j \in E$. Let $\sigma:E \sqcup \{j'\} \rightarrow E \sqcup \{j'\}$ be the trasposition of $j$ and $j'$.
        
        Since $e_j-e_{j'} \in \mathcal{V}'$, every support basis $B$ contains $j$ or $j'$ (or both). By symmetry, the support bases are uniquely determined by the support bases which contain $j'$.

        It follows from the definition that $B$ is a support basis of $\mathcal{V}$ if and only if $B \sqcup j'$ is a support basis of $\mathcal{V}'$. An RREF for $B \sqcup j'$ can be constructed from an RREF for $B$ as follows:
        \begin{itemize}
            \item if $\{v^i\}_{i \in B}$ is a RREF for $B$, and $j \in B$, then $\{v^i\}_{i \in B} \cup \{\sigma(v^j)\}$ is a RREF for $B \sqcup j'$.
            \item if $\{v^i\}_{i \in B}$ is a RREF for $B$, and $j \notin B$, then $\{v^i\}_{i \in B} \cup \{e_{j'}-e_j\}$ is a RREF for $B \sqcup j'$.
        \end{itemize}
        Since RREFs are unique, it follows that the RREFs in every support basis of $\mathcal{V}'$ is uniquely determined by the RREFs in every support basis of $\mathcal{V}'$. Since $k$-vector sets are defined in terms of their RREFs in every support basis, the duplication $\mathcal{V}'$ is unique.
    \end{proof}
\end{lemma}

\begin{lemma}\label{lem:parallel_rescale_to_duplicate}
	Let $k$ be a perfect idyll and let $\mathcal{V} \subset k^E$ be a $k$-vector set. Suppose that $i,j \in E$ are parallel in $\mathcal{V}$. Then there exists a $k$-vector set isomorphic to $\mathcal{V}$ by rescaling the $j$ coordinate in which $i$ and $j$ are duplicate.
    \begin{proof}
    	If $i,j \in E$ are parallel in $\mathcal{V}$. Then there exists a circuit
     $ae_i + b e_j \in \mathcal{V}$, with $a, b$ nonzero. By (isotropic) scaling, we also have $e_i + a^{-1}b e_j \in \mathcal{V}$.
     
     By rescaling the $j$ coordinate by $-a b^{-1}$, we obtain the desired $k$-vector set which contains $e_i - e_j$. By Lemma \ref{lem:vector_set_scaling}, this is isomorphic to $\mathcal{V}$.
    \end{proof}
\end{lemma}
For any $k$-vector set $\mathcal{V}$, by applying Lemma \ref{lem:parallel_rescale_to_duplicate} repeatedly, we obtain an isomorphic $k$-vector set in which all parallel elements are duplicate. Such a $k$-vector set has a unique simplification, since by symmetry the result is identical regardless of the choices of which parallel elements to delete.

\section{Proto-exact Categories}\label{sec:proto-exact}
Proto-exact categories are a generalization of exact categories, introduced by Dyckerhoff and Kapranov in \cite{DyckerhoffKapranov19}. Proto-exact categories contain a notion of short exact sequences, but are not necessarily additive.
\begin{definition}[Dyckerhoff-Kapranov, 2019 \cite{DyckerhoffKapranov19}]
    A category $\mathtt{C}$ is proto-exact by definition if it has a zero object and two classes of morphisms $\mathfrak{M}$ and $\mathfrak{E}$ (``\textbf{admissible monomorphisms}" and ``\textbf{admissible epimorphisms}") satisfying
    \begin{itemize}
        \item every morphism $0 \rightarrow A$ is in $\mathfrak{M}$ and every morphism $A \rightarrow 0$ is in $\mathfrak{E}$.
        
        \item $\mathfrak{M}$ and $\mathfrak{E}$ contain all isomorphisms and are closed under composition by isomorphisms.
        
        \item (denoting admissible monomorphisms in $\mathfrak{M}$ by $\hookrightarrow$ and admissible epimorphisms by $\mathfrak{E}$ $\twoheadrightarrow$) every commutative square 
        \[\begin{tikzcd}
        	A & B \\
        	C & D
        	\arrow[hook, from=1-1, to=1-2]
        	\arrow[two heads, from=1-1, to=2-1]
        	\arrow[two heads, from=1-2, to=2-2]
        	\arrow[hook, from=2-1, to=2-2]
        \end{tikzcd}\]
        is a push-out if and only if it is a pullback (in which case, it is called \textbf{bi-Cartesian}).

        \item Every diagram $C \twoheadleftarrow A \hookrightarrow B$ and $C \twoheadrightarrow D \hookleftarrow B$ can be completed to a bi-Cartesian square as above.
    \end{itemize}
    
    Given a proto-exact category, bi-Cartesian squares of the form
    \[\begin{tikzcd}
    	A & B \\
    	0 & C
    	\arrow[hook, from=1-1, to=1-2]
    	\arrow[two heads, from=1-1, to=2-1]
    	\arrow[two heads, from=1-2, to=2-2]
    	\arrow[hook, from=2-1, to=2-2]
    \end{tikzcd}\]
    are called admissible short-exact sequences.
    \end{definition}
    
    \begin{proposition}\label{prop:proto-exact_matroid_module}
        Let $k$ be a perfect idyll. Let $B$ be a band.
        \begin{itemize}
            \item The category $\texttt{Matroid}_\bullet$ is proto-exact, where $\mathfrak{M}$ are the restriction maps up to isomorphism, and $\mathfrak{E}$ are contraction maps up to isomorphism. (Theorem 5.11 in \cite{EppolitoJunSzczensny20}).
            
            \item The category $\texttt{Matoid}_k$ is proto-exact, where $\mathfrak{M}$ are the restriction maps up to isomorphism, and $\mathfrak{E}$ are contraction maps up to isomorphism (Theorem 3.11 in \cite{JunSistkoWright25}).
            
            \item The category $\texttt{Module}_B$ is proto-exact, where $\mathfrak{M}$ are the inclusion maps from a strict submodule up to isomorphism, and $\mathfrak{E}$ are quotient maps by a strict submodule up to isomorphism (Corollary 7.10 in \cite{Hamada26}). Additionally, any full subcategory of $\texttt{Module}_B$ which is closed under taking strict submodules and quotients by strict submodules is proto-exact (Theorem 7.9 in \cite{Hamada26}).
        \end{itemize}
    \end{proposition}

\section{Linear Spaces over Idylls and Hyperfields}\label{sec:linear_space}

Given that a module over a band $B$ is like a module over a commutative ring, we may expect that a module over a perfect idyll $k$ is like a vector space over a field. We may define linear independence as follows:

\begin{definition}\label{def:vector_independence}
    Let $k$ be a perfect idyll, and let $V$ be a $k$-module.
    
    A collection $\{v_i\}_{i \in S}$ of elements $v_i \in V$ is \textbf{(linearly) independent} if 
    $$\sum_{i \in S} a_i v_i \in N_V, \text{for $a_i \in k$}$$
    implies that $a_i = 0$, for all $i \in S$.
\end{definition}

However, there are pathologies with linear independence in $k$-modules. Unlike for vector spaces, the maximal linearly independent sets of $V$ do not necessarily form the bases of a matroid. In particular, their maximal independent sets may not be the same size, and there is no well-defined notion of dimension. The following counterexample over the Krasner hyperfield, in $\mathbb{K}^3$ is due to Chris Eppolito (Example 4.34 in \cite{Eppolito22}):

\begin{example}\label{eg:cartesian_fail}    
    Let $v_0, \dots, v_4 \in \mathbb{K}^3$ be the columns of the matrix
    $$\begin{bmatrix}
        0 & 1 & 1 & 1 & 1\\
        1 & 0 & 0 & 1 & 1\\
        1 & 0 & 1 & 0 & 1\\
    \end{bmatrix}$$
    The maximal linearly independent subsets are $012, 013, 04, 123, 124, 134$, which do not satisfy the basis exchange property, since they are not the same size.
\end{example}

To remedy the failure of linear independence in $k$-modules, we postulate an additional property:

\begin{definition}\label{def:k-linear_space}
    Let $k$ be a perfect idyll and let $V$ be a $k$-module. $V$ is a $k$-\textbf{\linearspace{}} if it additionally satisfies:

    \begin{itemize}
        \item For all finite collections $\{v_i\}_{i \in E}$ of elements $v_i \in V$, the set $$\mathcal{V}(\{v_i\}_{i \in E}) := \left\{ (a_i)_{i \in E} \in k^E \mid\sum_{i \in E} a_i v_i \in N_V \right\} $$ is a $k$-vector set $\subset k^E$ (according to Definition \ref{def:vector_set}).
    \end{itemize}

    $k$-\linearspace{}s with linear maps form a full subcategory of $\texttt{Module}_k$ which we denote by $\texttt{LinearSpace}_k$.
\end{definition}
The motivation behind this definition is that when $V$ is a vector space over a field $k$, and $\{v_i\}_{i \in E}$ is a finite collection of vectors in $V$, there is a canonical linear map $k^E \rightarrow V$ sending the $i$-th coordinate vector to $v_i$. Then, $\mathcal{V}(\{v_i\}_{i \in E})$ is the kernel of this map, and is a linear subspace of $k^E$. By Proposition 2.19 in \cite{Anderson19}, $k$-vector sets naturally generalize linear subspaces of $k^E$ to the case where $k$ is a perfect idyll. 

By construction, for a subset $S \subset E$, $\mathcal{V}(\{v_i\}_{i \in S})$ is the restriction $\mathcal{V}(\{v_i\}_{i \in E})_{|S}$. Implicit in the requirement that both of these are $k$-vector sets is that $k$-vector sets must be compatible the with the restriction operation, hence the requirement that the idyll $k$ be perfect.

We shall see that the property that $\mathcal{V}(\{v_i\}_{i \in E})$ is a $k$-vector set for all finite collections $\{v_i\}_{i \in E}$ is sufficient to ensure that linear independence in $k$-\linearspace{}s satisfies matroid independence axioms.

\begin{proposition}
    Let $k$ be a perfect idyll, and $V$ be a $k$-module, and $\{v_i\}_{i \in E}$ a finite collection of elements $v_i \in V$.
    
    The maximal linearly independent subsets of $\{v_i\}_{i \in E}$ are exactly complementary to the support bases of the $k$-vector set $\mathcal{V}(\{v_i\}_{i \in E})$ (defined in \ref{def:vector_set}). In other words, $\{v_i\}_{i \in S}$ is linearly independent if and only if $S$ is contained in the complement of a support basis $\mathcal{V}(\{v_i\}_{i \in E})$.
    \begin{proof}
        The contrapositive statement follows from the definitions of independence and support basis:
        
        If $\{v_i\}_{i \in S}$ is linearly dependent, then there exists a nonzero $\sum_{i \in E} a_i v_i \in N_V$ with $(a_i)_{i \in E}$ supported on $S$ and so each support basis of $\mathcal{V}(\{v_i\}_{i \in E})$ intersects $\mathrm{supp}((a_i)_{i \in E}) \subset S$.
        
        Conversely, if $S$ intersects each support basis $B$, there exists a nonzero $\sum_{i \in E} a_i v_i \in N_V$ supported on $S$. (Otherwise, $B \setminus S \subsetneq B$ supports $\mathcal{V}(\{v_i\}_{i \in E})$, contradicting minimality of the support basis.) Therefore, $\{v_i\}_{i \in S}$ is linearly dependent.
    \end{proof}
\end{proposition}

\begin{proposition}\label{prop:linear_independence_is_matroidal}
    Let $k$ be a perfect idyll, and $V$ a $k$-\linearspace{} and $\{v_i\}_{i \in E}$ a finite collection of elements $v_i \in V$.
    
    The maximal independent subsets of $\{v_i\}_{i \in E}$ are the bases of a matroid. (i.e., they satisfy the basis exchange property.) In particular, all maximal independent subsets have the same size.
    
    \begin{proof}
        This follows from Lemma 2.13 in \cite{Anderson19} which states that if $\mathcal{V}(\{v_i\}_{i \in E})$ has the structure of a $k$-vector set, then the support bases form the bases of a matroid. Then, the maximal linearly independent subsets which are the complement of the support bases form the bases of the dual matroid.
    \end{proof}
\end{proposition}

\begin{definition}\label{def:finitely_generated}
    Let $k$ be a perfect idyll. A $k$-module (or $k$-\linearspace{}) $V$ is \textbf{finitely generated} if there exists a finite collection $\{e_i\}_{i \in E}$ of elements $e_i \in V$, such that for each $v \in V$, there exists $a \in k$ and $i \in E$ such that $v = ae_i$). We call such a set $\{e_i\}_{i \in E}$ a \textbf{generating set}.
\end{definition}

\begin{proposition}\label{prop:generator_isomorphism}
    Let $k$ be a perfect idyll, and let $V$ be a finitely generated $k$-\linearspace{}. Suppose $\{e_i\}_{i \in E}$ and $\{e_i'\}_{i \in E'}$ are generating sets whose elements are nonzero and pairwise non-parallel. Then $\mathcal{V}(\{e_i\}_{i \in E})$ and $\mathcal{V}(\{e_i'\}_{i \in E'})$ are isomorphic as $k$-vector sets.
    \begin{proof}
        By definition, for each $i \in E$, there exists $a \in k$ and $j \in E'$ with $e_i = ae_j'$. Since generators are nonzero and non-parallel, this establishes a bijection between $(\{e_i\}_{i \in E}$ and $\{e_i'\}_{i \in E'}$, and they differ only by scaling by nonzero constant and reordering of elements (a bijection $E \rightarrow E'$). Then the $k$-vector sets $\mathcal{V}(\{e_i\}_{i \in E})$ and $\mathcal{V}(\{e_i'\}_{i \in E'})$ differ only by scaling or permuting the coordinate directions, and by Lemma \ref{lem:vector_set_scaling} and Lemma \ref{lem:vector_set_permutation}, they are isomorphic.
    \end{proof}
\end{proposition}

\begin{proposition}\label{prop:generator_loop_parallel}
    Let $k$ be a perfect idyll, and let $V$ be a $k$-\linearspace{} finitely generated by $\{e_i\}_{i \in E}$. 
    \begin{itemize}
        \item The element $i \in E$ is a loop in the $k$-vector set $\mathcal{V}(\{e_i\}_{i \in E})$ if and only if $e_i = 0$.
        \item The elements $i,j \in E$ are parallel in the $k$-vector set $\mathcal{V}(\{e_i\}_{i \in E})$ if and only if $e_i$ and $e_j$ are parallel, i.e., there exists $a, b \in k$ such that $a e_i = b e_j$.
    \end{itemize}
    \begin{proof}
        It is straightforward to verify equivalence of definitions:
        \begin{itemize}
            \item The element $i \in E$ is a loop, if and only if there exists nonzero $a \in k$ with $ae_i \in N_V$. By uniqueness of opposite elements, this occurs if and only if $ae_i = -0$, and hence $e_i=-a^{-1}0 = 0$.
            \item The elements $i,j \in E$ are parallel if and only if there exists $a_i, a_j \in k$ with $a_ie_i+a_je_j \in N_V$. By uniqueness of opposite elements, this occurs if and only if $a_ie_i = -a_je_j$.
        \end{itemize}
    \end{proof}
\end{proposition}
As a result, the elements $\{e_i\}_{i \in E}$ are nonzero and pairwise non-parallel if and only if $\mathcal{V}(\{e_i\}_{i \in E})$ is a simple $k$-vector set.
\begin{proposition}\label{prop:generator_simplification}
    Let $k$ be a perfect idyll, and let $V$ be a $k$-\linearspace{} finitely generated by $\{e_i\}_{i \in E}$. Suppose $S \subset E$ is a subset so that $\{e_i\}_{i \in S}$ is a finite generating set whose elements are nonzero and pairwise nonparallel.
    Then, the $k$-vector set $\mathcal{V}(\{e_i\}_{i \in S})$ is a simplification of the $k$-vector set $\mathcal{V}(\{e_i\}_{i \in E})$.
    \begin{proof}
        It is straightforward to verify equivalence of definitions. The generating set $\{e_i\}_{i \in S}$ differs from $\{e_i\}_{i \in E}$ by removing any zero elements and all but one element from each parallel class. Then $\mathcal{V}(\{e_i\}_{i \in S})$ differs from $\mathcal{V}(\{e_i\}_{i \in E})$ by deleting all loops and all but one of each parallel class, which by definition is the simplification.
    \end{proof}
\end{proposition}

\begin{proposition}\label{prop:linear_space_specification_by_generating_set}
    Let $V$ be $k$-\linearspace{}, finitely generated by $\{e_i\}_{i \in E}$. Suppose the generators $\{e_i\}_{i \in E}$ are nonzero and pairwise non-parallel. Then $V$ is uniquely determined by the simple $k$-vector set $\mathcal{V}(\{e_i\}_{i \in E})$.
    \begin{proof} 
        Since the generators are are nonzero and pairwise non-parallel, the underlying pointed set of $V$ is
        $$\{ae_i \mid i\in E, a \in k\}.$$
        Therefore, we are done if we can show that the null set $N_V$ is uniquely determined by $\mathcal{V}(\{e_i\}_{i \in E})$, which is simple again since the generators are are nonzero and pairwise non-parallel.

        Since $\{e_i\}_{i \in E}$ is a generating set, any formal sum $\sum v_i \in V^+$ may be expressed in the form
        $\sum v_i = \sum_{i \in E'} a_ie_i$
        where the $e_i$ terms are possibly duplicated (i.e., $E'$ is a multiset which is supported on $E$). Since the generators are are nonzero and pairwise non-parallel, this expression is unique.

        By definition, $\sum v_i \in N_V$ if and only if $(a_i)_{i \in E'} \in \mathcal{V}(\{e_i\}_{i \in E'})$. It follows from the definition that the $k$-vector set $\mathcal{V}(\{e_i\}_{i \in E'})$ is a duplication of $\mathcal{V}(\{e_i\}_{i \in E})$, where each element $i \in E$ is duplicated up to the number of times it appears in the multiset $E'$ (where duplication is in the sense of Definition \ref{def:vector_set_duplication}).

        By Lemma \ref{lem:vector_set_duplication_unique}, the duplication is uniquely determined by $\mathcal{V}(\{e_i\}_{i \in E})$, and therefore so is the null set $N_V$, and $V$ itself.
    \end{proof}
\end{proposition}

\begin{definition}\label{def:linear_space_specification_by_generating_set}
    In light of the unique specification determined by Proposition \ref{prop:linear_space_specification_by_generating_set}, we specify a $k$-\linearspace{} $V$ which is finitely generated by nonzero and pairwise non-parallel generators $\{e_i\}_{i \in E}$, and with null set determined by the $k$-vector set $\mathcal{V} =\mathcal{V}(\{e_i\}_{i \in E})$
    by
    $$V := k\langle e_i \mid i \in E\rangle / \left\langle \sum_{i \in E} a_ie_i \mid (a_i)_{i \in E} \in \mathcal{V}\right\rangle$$
    Since $k$-vector sets $\mathcal{V}$ themselves are uniquely determined by their circuits $\mathcal{C} \subset \mathcal{V}$ (i.e., elements of nonzero minimal support), we may furthermore write 
    $$V := k\langle e_i \mid i \in E\rangle / \left\langle \sum_{i \in E} a_ie_i \mid (a_i)_{i \in E} \in \mathcal{C}\right\rangle$$
    The quotient notation is thus chosen so that $$\left\langle \sum_{i \in E} a_ie_i \mid (a_i)_{i \in E} \in \mathcal{C}\right\rangle = \left\langle \sum_{i \in E} a_ie_i \mid (a_i)_{i \in E} \in \mathcal{V}\right\rangle$$ is the nullkernel of the unique surjective map from the free $k$-\linearspace{} $k\langle e_i \mid i \in E\rangle \rightarrow V$ which sends $e_i \mapsto e_i$.
\end{definition}

\begin{example}
    For example, we specify the $\mathbb{K}$-\linearspace{} generated by $e_0, e_1, e_2$, with simple $k$-vector set $\mathcal{V}(\{e_0, e_1, e_2\}) = \{000, 111\}$ by
    $$\frac{\mathbb{K}\langle e_0, e_1, e_2\rangle}{\langle e_0 + e_1 + e_2\rangle}.$$
    The duplication at index $2$ is a (not simple) $k$-vector set $\mathcal{V}(\{e_0, e_1, e_2, e_{2}\}) = \{0000, 1110,1101,0011, 1111\}$, and hence the null ideal $\langle e_0 + e_1 + e_2\rangle$ contains elements such as $e_0 + e_1 + e_2 + e_2$.
\end{example}

\begin{definition}\label{def:linear_space:dimension}
    Let $k$ be a perfect idyll, and $V$ a finitely generated $k$-\linearspace{}.
    
    Then define the \textbf{dimension} $\dim V$ to be the size of any maximal linearly independent subset of the generating set $\{e_i\}_{i \in E} \subset V$. This is well-defined by Proposition \ref{prop:linear_independence_is_matroidal}. By Proposition \ref{prop:generator_isomorphism} and Proposition \ref{prop:generator_simplification}, the dimension is independent of the choice of generating set.
\end{definition}
We denote the full subcategory in $\texttt{LinearSpace}_k$ of finitely generated $k$-\linearspace{} by $\texttt{LinearSpace}^{\mathrm{f.g.}}_k$. By the previous paragraph, finitely generated implies finite dimensional.

\begin{example}
    \noindent
    \begin{itemize}
    \item $\{0\}$ with trivial null set is a $k$-\linearspace{}. It is the zero object in $\texttt{LinearSpace}_k$.
    
    \item Any idyll $k$ is a $1$ dimensional $k$-\linearspace{}, where scalar multiplication is the usual multiplication.
    \end{itemize}
\end{example}

As in $\texttt{Module}_B$, we may form strict submodules and quotients by strict submodules of $k$-\linearspace{}s, which are again $k$-\linearspace{}s:
\begin{proposition}\label{prop:linear_space_subobject}
    Let $k$ be a perfect idyll, and $V$ a $k$-\linearspace{}. Let $W \subset V$ be a strict submodule. Then, 
    \begin{itemize}
        \item the strict submodule $W$ is a $k$-\linearspace{}.
        \item the quotient by strict submodule $V/\langle W\rangle$ is a $k$-\linearspace{}.
    \end{itemize}
    Also, if $V$ is finitely generated, then so are $W$ and $V/\langle W\rangle$.
    
    \begin{proof}
       We verify the vector set property. Any finite collection $\{v_i\}_{i \in E} \subset W$, is also a finite collection $\subset V$. Hence, $\mathcal{V}(\{v_i\}_{i \in E})$ is a $k$-vector set.

       For any finite collection $\{v'_i\}_{i \in E} \subset V/\langle W \rangle$, choose any representative $v_i \in V$ of the equivalence class, such that $v_i' = [v_i]$. Then since scalar multiplication is well defined on equivalence classes, it follows from the definition of quotient that for all $a_i \in k$,
       $\sum a_iv'_i = \sum [a_i v_i] \in N_{V/\langle W \rangle}$ if and only if $\sum a_i v_i \in N_V$. Hence, $\mathcal{V}(\{v'_i\}_{i \in E}) = \mathcal{V}(\{v_i\}_{i \in E})$ is a $k$-vector set.

       If $V$ is finitely generated by $\{v_i\}_{i \in E}$, then $W$ is finitely generated by the subset $W \cap \{v_i\}_{i \in E}$, and $V/\langle W\rangle$ is finitely generated by the equivalence classes $\{[v_i]\}_{i \in E}$.
    \end{proof}
\end{proposition}

\begin{proposition}
    The categories $\texttt{LinearSpace}_k$ and $\texttt{LinearSpace}^{\mathrm{f.g.}}_k$ are proto-exact, where $\mathfrak{M}$ are the inclusion maps from a strict submodule up to isomorphism, and $\mathfrak{E}$ are quotient maps by a strict submodule up to isomorphism.
    \begin{proof}
        Since $\texttt{LinearSpace}_k$ and $\texttt{LinearSpace}^{\mathrm{f.g.}}_k$ are a full subcategories of $\texttt{Module}_k$, and closed under taking strict submodules and quotient by strict submodules, this follows from the third point of Proposition \ref{prop:proto-exact_matroid_module} (Theorem 7.9 in \cite{Hamada26}).
    \end{proof}
\end{proposition}

Next, we investigate some other categorical properties of $\texttt{LinearSpace}_k$. We will see that $\texttt{LinearSpace}_k$ has equalizers, kernels, cokernels, and coproducts, but unlike $\texttt{Module}_k$, has no coequalizers and products in general. This is inevitable due to an equivalence with the category of matroids, which also has no coequalizers and products.

\begin{proposition}
    Let $k$ be a perfect idyll and $f,g:V \rightarrow W$ be linear maps of $k$-\linearspace{}s $V$ and $W$.
    The equalizer $\mathrm{eq}(f,g)$, kernel $\mathrm{ker}(f)$ and cokernel $\mathrm{coker}(f)$ are also $k$-\linearspace{}s.
    \begin{proof}
        Since the equalizer $\mathrm{eq}(f,g)$ and kernel $\mathrm{ker}(f)$ are strict submodules of $V$, and the cokernel is the quotient of $W$ by a strict submodule, this is a immediate corollary of Proposition \ref{prop:linear_space_subobject}.
    \end{proof}
\end{proposition}
Notably, the coequalizer is not the quotient by a strict submodule, and it does not follow that the coequalizer is a $k$-\linearspace{}. The equalizer, kernel, and cokernel satisfy the relevant universal properties in $\texttt{LinearSpace}_k$ and $\texttt{LinearSpace}^{\mathrm{f.g.}}_k$, since they do in $\texttt{Module}_k$. 

\begin{proposition}
    Let $k$ be a perfect idyll, and $\{V_i\}_{i \in I}$ a family of $k$-\linearspace{}s indexed by $I$. The direct sum $\bigoplus_{i \in I} V_i$ is a $k$-\linearspace{}.
    \begin{proof}
        We verify the vector set property. Any finite collection $\{v_j\}_{j \in E} \subset \bigoplus_{i \in I} V$ may be partitioned by the disjoint union into finite collections $\{v_j\}_{j \in E_i} \subset V_i$ for $i \in I$ where $\bigsqcup_{i \in I} E_i = E$. Since $E$ is finite, only finitely many of $E_i$'s are nonempty. By hypothesis, $\mathcal{V}(\{v_j\}_{j \in E_i}) \subset k^{E_i}$ is a $k$-vector set for all $i \in I$.

        By definition of the direct sum, if $\sum_{j \in E} a_j v_j \in N_{\bigoplus_{i \in I}V_i}$, then $\sum_{j \in E_i} a_j v_j \in N_{V_i}$ for all $i \in I$. Therefore, $\mathcal{V}(\{v_j\}_{j \in E}) = \bigoplus_{i \in I}\mathcal{V}(\{v_j\}_{j \in E_i})$ is the $k$-vector set constructed by the direct sum of $k$-vector sets $\mathcal{V}(\{v_j\}_{j \in E_i})$. This is a finite direct sum since only finitely many of $E_i$'s are nonempty, and the binary direct sum in Definition \ref{def:vector_set_direct_sum} may be extended by induction.
    \end{proof}
\end{proposition}
The direct sum satisfies the universal property of the coproduct in $\texttt{LinearSpace}_k$ since it does in $\texttt{Module}_k$. $\texttt{LinearSpace}^{\mathrm{f.g.}}_k$ only has finite direct sums.

\begin{definition}
    For any set $S$, $k\langle S \rangle := k^{\oplus S}$ is the free $k$-\linearspace{} generated by $S$.
\end{definition}
    This satisfies the universal property of free object in $\texttt{LinearSpace}_k$.
\begin{proof}
    Let $V$ be a $k$-\linearspace{}, and $f:S \rightarrow V$ a set map. For each element $e \in S$, there exists a unique linear map $k \rightarrow V$ mapping $1 \rightarrow f(e)$, and hence the universal property of the coproduct guarantees a unique map $k\langle S \rangle \rightarrow V$.
\end{proof}

\begin{proposition}
    Let $k$ be a perfect idyll and let $V$ and $W$ be $k$-\linearspace{}s. In general, $\mathrm{Hom}(V, W)$ is \textbf{not} a $k$-\linearspace{}.
    \begin{proof}
        We give a counterexample over the Krasner hyperfield $k = \mathbb{K}$. Let $V = \mathbb{K}^{\oplus 3} = \mathbb{K}\langle e_0, e_1, e_2 \rangle$

        The linear maps $\mathbb{K}^{\oplus 3} \rightarrow \mathbb{K}$ are determined by the images of $e_0, e_1, e_2$. As a result, $\mathrm{Hom}(\mathbb{K}^{\oplus 3}, \mathbb{K})$ is isomorphic to the product $\mathbb{K}^3$, which is not a $k$-\linearspace{}, as shown in Example \ref{eg:cartesian_fail}.
    \end{proof}
\end{proposition}
One special exception is when $V=k$; then $\mathrm{Hom}(k, W)$ is isomorphic to $W$ and is a $k$-\linearspace{}.

Finally, we examine the relationship between the categories $\texttt{Matroid}^{\mathrm{simple}}_\bullet$, $\texttt{Matroid}^{\mathrm{simple}}_k$, and $\texttt{LinearSpace}^{\mathrm{f.g.}}_k$. The fact that $\texttt{LinearSpace}_k$ has no coequalizers and no products in general follows from a categorical equivalence with $\texttt{Matroid}^{\mathrm{simple}}_\bullet$, since it is known that $\texttt{Matroid}^{\mathrm{simple}}_\bullet$ has no coequalizers and no products.

\begin{proposition}\label{prop:categorical_equivalence}
    The category of finitely generated $\mathbb{K}$-\linearspace{}s with linear maps ($\texttt{LinearSpace}^{\mathrm{f.g.}}_\mathbb{K}$) is equivalent to the category of simple matroids with pointed strong maps ($\texttt{Matroid}^{\mathrm{simple}}_\bullet$).
    \begin{proof}
        Since $\mathbb{K}$ is finite, every finitely generated $\mathbb{K}$-\linearspace{} is finite.
        Define a functor $F: \texttt{LinearSpace}^{\mathrm{f.g.}}_\mathbb{K} \rightarrow \texttt{Matroid}^{\mathrm{simple}}_\bullet$

        \begin{itemize}
            \item on finitely generated $\mathbb{K}$-\linearspace{}s $V$ by $F(V)$ is the matroid on $E = V \setminus 0$, where independent subsets are the linearly independent subsets in $V \setminus 0$.
            
            \item on $\mathbb{K}$-linear maps $f: V_1 \rightarrow V_2$ by $F(f): F(V_1) \rightarrow F(V_2)$ for $v \in V_1 \setminus 0$
            
            $$v \mapsto \begin{cases}
                f(v) &\text{if $f(v) \neq 0$}\\
                \bullet &\text{if $f(v) = 0$}
            \end{cases}$$
        \end{itemize}
        $F(V)$ is a well-defined matroid by Proposition \ref{prop:linear_independence_is_matroidal}. Since $F(V)$ omits $0$ and there is only one nonzero element in the $\mathbb{K}$-span of each nonzero vector, $F(V)$ is simple.
        
        It is clear that set functions $V_1 \rightarrow V_2$ which send $0 \mapsto 0$ are in bijection via $F$ with point preserving functions $F(V_1) \sqcup  \bullet \rightarrow F(V_2) \sqcup \bullet$, with $\bullet$ taking the place of $0$.

        Then for set functions $f:V_1 \rightarrow V_2$
        \begin{itemize}
            \item[\:] $f$ is a linear map
            \item[$\iff$] $f$ preserves null set $f(N_{V_1}) \subset N_{V_2}$
            \item[$\iff$] $f$ preserves the $\mathbb{K}$-vector sets $\mathcal{V}(V_1 \setminus 0) \subset \mathcal{V}(\{f(v) \mid v \in V_1 \setminus 0\})$
            \item[$\iff$] it sends cycles in $F(V_1)$ to cycles in $F(V_2)$ up to simplification by $F(f)$,
            \item[$\iff$] $F(f)$ is a strong map (by Lemma \ref{lem:strong_map_circuits}).
        \end{itemize}
        Therefore, the functor $F$ is well-defined and fully faithful.

        Given any simple matroid $M$ on ground set $E$, it is isomorphic to $F(V)$ for the $\mathbb{K}$-\linearspace{} $V$ given by
        $$V := \mathbb{K}\langle e_i \mid i \in E\rangle / \left\langle \sum_{i \in S} e_i \mid \text{$S$ is a cycle in $M$}\right\rangle$$.
        
        Since $M$ is simple, and has no cycles of size $2$ or $1$, no elements are identified when passing to the quotient, so $F(V)$ has the same number of elements as $E$, with canonical bijection $i \mapsto e_i$. By construction $\mathcal{V}\{e_i \mid i \in E\}$ is the $\mathbb{K}$-vector set indicating the cycles of $M$, so the maximal linearly independent subsets of $F(V) = \{e_i \mid i \in E\}$ are complementary to support bases of the $\mathbb{K}$-vector set, and identical to the bases of $M$. Hence the bijection $i \mapsto e_i$ is an isomorphism of matroids, and $F$ is essentially surjective.

        Since $F$ is fully faithful and essentially surjective, it is an equivalence of categories $\texttt{LinearSpace}^{\mathrm{f.g.}}_\mathbb{K} \rightarrow \texttt{Matroid}^{\mathrm{simple}}_\bullet$.
    \end{proof}
\end{proposition}

\begin{proposition}
    The $\texttt{LinearSpace}_k$ does not have all coequalizers, or products in general.
    \begin{proof}
        In \cite{HeunenPatta18}, Proposition 3.5 and 3.7, the Heunen and Patta show that $\texttt{Matroid}^{\mathrm{simple}}_\bullet$ has no coequalizers or products. Hence by the categorical equivalence in Proposition \ref{prop:categorical_equivalence}, neither does the category $\mathbb{K}$-\linearspace{}s.
    \end{proof}
\end{proposition}
The lack of products in the category of $k$-\linearspace{}s provides an explanation to why naive linear algebra in $k^n$ fails, as in Example \ref{eg:cartesian_fail}.

\begin{proposition}\label{prop:matroid_linear_space_embedding}
    Let $k$ be a perfect idyll. There is a faithful embedding of the category $\texttt{Matroid}^{\mathrm{simple}}_k$ into the category $\texttt{LinearSpace}^{\mathrm{f.g.}}_k$
    (i.e., a functor $\texttt{Matroid}^{\mathrm{simple}}_k \rightarrow \texttt{LinearSpace}^{\mathrm{f.g.}}_k$ which is faithful and essentially injective and surjective on objects).
    \begin{proof}
        Define a functor $F: \texttt{Matroid}^{\mathrm{simple}}_k \rightarrow \texttt{LinearSpace}^{\mathrm{f.g.}}_k$
        \begin{itemize}
            \item on simple $k$-vector sets $\mathcal{V} \subset k^E$ by 
            $$F(\mathcal{V}) = k\langle e_i \mid i \in E \rangle/\left\langle \sum_{i \in E} a_i e_i \mid (a_i)_{i\in E}\in \mathcal{V} \right\rangle$$
            as specified by Definition \ref{def:linear_space_specification_by_generating_set}. By construction, $\{e_i\}_{i \in E}$ is a finite generating set for $F(\mathcal{V})$ and $\mathcal{V} = (\{e_i\}_{i \in E})^{\perp}$.
            
            \item on morphisms of $k$-vector sets $f:\mathcal{V}_1\subset k^{E_1} \rightarrow \mathcal{V}_2 \subset k^{E_2}$ by $F(f): F(\mathcal{V}_1) \rightarrow F(\mathcal{V}_2)$ which maps, for $i \in E_1$
            $$e_i \mapsto \begin{cases}
                f_{ji}e_j & \text{$f_{ji}$ is the nonzero entry in the $i$-th column}\\
                0 & \text{if the $i$-th column is zero}
            \end{cases}$$
            
            Since $\mathcal{V}$ is simple, no elements of $k\langle e_i \mid i \in E \rangle$ are identified when passing to the quotient, and $F(f)$, defined on generators, extends uniquely to a map which is compatible with scalar multiplication.

            If $\sum a_ie_i \in N_{F(\mathcal{V}_1)}$, 
            then $(a_i)_{i \in E_1} \in \mathcal{V}_1$ and $(f_{ji}a_i)_{j \in E_2} =f((a_i)_{i \in E_1})\in f(\mathcal{V}_1) \subset \mathcal{V}_2$ so $\sum F(f)(a_ie_i) = \sum f_{ji}a_i e_j \in N_{F(\mathcal{V}_1)}$, so $F(f)$ is a well defined linear map.
            
            Since the $F(f)$ is uniquely specified by the value on the generators $e_i$, it is uniquely specified by the submonomial matrix for $f$, and $F$ is faithful.
            
            For any finitely generated $k$-\linearspace{} $V$, there exists a finite generating set $\{e_i\}_{i \in E}$ so that the $e_i$'s are nonzero and pairwise non-parallel. Then $\mathcal{V}=\mathcal{V}(\{e_i\}_{i \in E}) \subset k^E$ is a simple $k$-vector set, with $V \cong F(\mathcal{V})$. Therefore, $F$ is essentially surjective.
            
            For any simple $k$-vector sets $\mathcal{V}_1 \subset k^{E_1}$, $\mathcal{V}_2 \subset k^{E_2}$, suppose there is an isomorphism $\iota: F(\mathcal{V}_1) \rightarrow F(\mathcal{V}_2)$. The generators $e_i$ may be transported along the isomorphism, and are nonzero and pairwise nonparallel since $\mathcal{V}_1$ and $\mathcal{V}_2$ are simple. By Proposition \ref{prop:generator_isomorphism}, the vector sets are isomorphic, $\mathcal{V}_1 \cong \mathcal{V}_2$, and the functor $F$ is essentially injective.
        \end{itemize}
    \end{proof}
\end{proposition}
The embedding functor $F:\texttt{Matroid}^{\mathrm{simple}}_k \rightarrow \texttt{LinearSpace}^{\mathrm{f.g.}}_k$ is not full. Since submonomial matrices have at most one nonzero entry per column, the essential image contains only the linear maps $f:V_1 \rightarrow V_2$ for which, for each nonzero $v_2 \in V_2$, there exists at most one $v_1 \in V_1$ with $f(v_1) = v_2$.

The following diagram summarizes the faithful embeddings of categories; the solid arrows indicate full subcategories, the dashed arrows indicate embeddings which are not full, and $\simeq$ indicates categorical equivalence: 

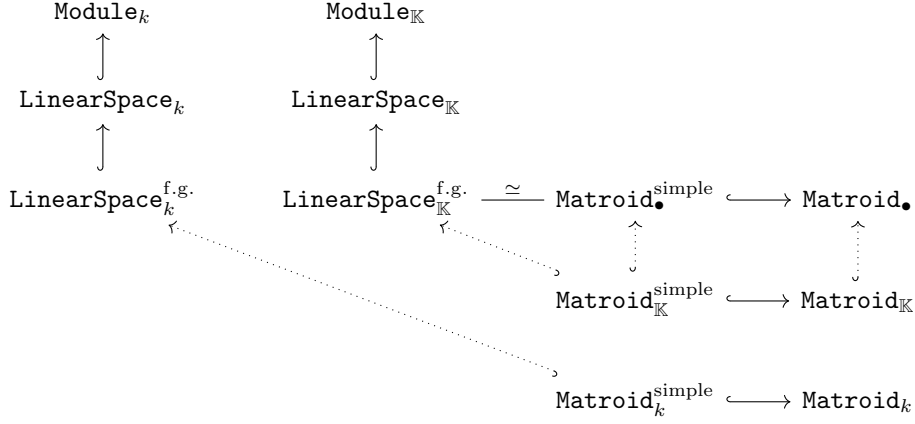
\[\begin{tikzcd}
	{\texttt{Module}_k} & {\texttt{Module}_\mathbb{K}} && \\
	{\texttt{LinearSpace}_k} & {\texttt{LinearSpace}_\mathbb{K}} \\
	{\texttt{LinearSpace}^{\mathrm{f.g.}}_k} & {\texttt{LinearSpace}^{\mathrm{f.g.}}_\mathbb{K}} & {\texttt{Matroid}^{\mathrm{simple}}_\bullet} & {\texttt{Matroid}_\bullet} \\
	&& {\texttt{Matroid}^{\mathrm{simple}}_\mathbb{K}} & {\texttt{Matroid}_\mathbb{K}} \\
	&& {\texttt{Matroid}^{\mathrm{simple}}_k} & {\texttt{Matroid}_k}
	\arrow[hook, from=2-1, to=1-1]
	\arrow[hook, from=2-2, to=1-2]
	\arrow[hook, from=3-1, to=2-1]
	\arrow[hook, from=3-2, to=2-2]
	\arrow["\simeq", no head, from=3-2, to=3-3]
	\arrow[hook, from=3-3, to=3-4]
	\arrow[dotted, hook, from=4-3, to=3-2]
	\arrow[dotted, hook, from=4-3, to=3-3]
	\arrow[hook, from=4-3, to=4-4]
	\arrow[dotted, hook, from=4-4, to=3-4]
	\arrow[dotted, hook, from=5-3, to=3-1]
	\arrow[hook, from=5-3, to=5-4]
\end{tikzcd}\]

Lastly, we show that all $k$-vector sets (not only simple ones) are realized by vector arrangements in $k$-\linearspace{}s, as long as the arrangement is allowed to contain duplicates and zero.
\begin{definition}
	Let $k$ be a perfect idyll, $E$ a finite set, $\mathcal{V} \subset k^E$ be a $k$-vector set, $V$ be a $k$-\linearspace{}, and $\{e_i\}_{i \in E}$ a finite collection of elements $e_i \in V$. We call $\{e_i\}_{i \in E}$ a \textbf{vector arrangement} and say that $\mathcal{V}$ is \textbf{realized} by or \textbf{represented} by $\{e_i\}_{i \in E}$ if $\mathcal{V} = \mathcal{V}(\{e_i\}_{i \in E})$.
\end{definition}

\begin{lemma}\label{lem:realize_rescaling}
	Let $k$ be a perfect idyll and $V$ be a $k$-\linearspace{}. Suppose that a $k$-vector set $\mathcal{V}\subset k^E$ is realized by $\{e_i\}_{i \in E}$. Let $\lambda = (\lambda_i)_{i \in E}$ be a set of nonzero non-isotropic scaling constants. Then $\lambda(\mathcal{V})$ is realized by $\{\lambda_i^{-1} \cdot e_i\}_{i \in E}$.

    \begin{proof}
    The following are equivalent by definition:
    \begin{align*}
    (a_i)_{i \in E} &\in \mathcal{V}(\{e_i\}_{i \in E})\\
    \sum a_i \lambda_i \lambda^{-1}_i e_i = \sum a_i e_i &\in N_V\\
    (\lambda_i a_i)_{i \in E} &\in \mathcal{V}(\{\lambda^{-1}_i e_i\}_{i \in E})
    \end{align*}
    Therefore, $\lambda(\mathcal{V})$ is realized by $\{\lambda_i^{-1} \cdot e_i\}_{i \in E}$.
    \end{proof}
\end{lemma}

\begin{proposition}
	Let $k$ be a perfect idyll. All $k$-vector sets are realized by a finite collection of elements in a $k$-\linearspace{}.

    \begin{proof}
    	Each $k$-vector set $\mathcal{V} \subset k^E$ is isomorphic via non-isotropic scaling to a $k$-vector set $\mathcal{V}' \subset k^E$ in which all parallel elements are duplicate. Let $P$ be the non-loop parallel classes, and $L \subset E$ be the set of loops. Together, the subsets $P \cup \{L\}$ partition $E$.
        By symmetry $\mathcal{V}'$ has a unique simplification $\mathrm{simp}(\mathcal{V}') \subset k^P$.
    
        The simplification $\mathrm{simp}(\mathcal{V}')$ is realized by the generators $\{e_i\}_{i \in P}$ in the \linearspace{} 
        $$V = k\langle e_i \mid i \in P \rangle/\left\langle \sum_{i \in P} a_i e_i \mid (a_i)_{i \in P} \in \mathrm{simp}(\mathcal{V}') \right\rangle.$$
        i.e., the one given by the embedding $\texttt{Matroid}^{\mathrm{simple}}_k \rightarrow \texttt{LinearSpace}^{\mathrm{f.g.}}_k$ in Proposition \ref{prop:matroid_linear_space_embedding}.
    
        Then $\mathcal{V}'$ is also realized in $V$ by $$\{e_{[i]}\}_{i \in E \setminus L} \cup \{0\}_{i \in L}$$ where $[i]$ indicates the parallel class of $i$. Thus, the element $e_{[i]}$ appears duplicated up to the number of elements in its parallel class and $0$ appears once for each loop.
        
        Then $\mathcal{V}$ is realizable by Lemma \ref{lem:realize_rescaling}, since it is isomorphic to $\mathcal{V}'$ by rescaling.
    \end{proof}
\end{proposition}

\bibliographystyle{plain} 
\bibliography{refs}

@article{BakerJinLorscheid25,
author = {Baker, Matthew and Jin, Tong and Lorscheid, Oliver},
title = {New building blocks for $\mathbb{F}_1$-geometry: Bands and band schemes},
journal = {Journal of the London Mathematical Society},
volume = {111},
number = {4},
pages = {e70125},
doi = {https://doi.org/10.1112/jlms.70125},
url = {https://londmathsoc.onlinelibrary.wiley.com/doi/abs/10.1112/jlms.70125},
eprint = {https://londmathsoc.onlinelibrary.wiley.com/doi/pdf/10.1112/jlms.70125},
year = {2025}
}

@article{BakerLorscheid21,
title = {The moduli space of matroids},
journal = {Advances in Mathematics},
volume = {390},
pages = {107883},
year = {2021},
issn = {0001-8708},
doi = {https://doi.org/10.1016/j.aim.2021.107883},
url = {https://www.sciencedirect.com/science/article/pii/S0001870821003224},
author = {Matthew Baker and Oliver Lorscheid}
}

@article{BakerBowler19,
title = {Matroids over partial hyperstructures},
journal = {Advances in Mathematics},
volume = {343},
pages = {821-863},
year = {2019},
issn = {0001-8708},
doi = {https://doi.org/10.1016/j.aim.2018.12.004},
url = {https://www.sciencedirect.com/science/article/pii/S0001870818304961},
author = {Matthew Baker and Nathan Bowler}
}

@article{DressWenzel92,
title = {Valuated matroids},
journal = {Advances in Mathematics},
volume = {93},
number = {2},
pages = {214-250},
year = {1992},
issn = {0001-8708},
doi = {https://doi.org/10.1016/0001-8708(92)90028-J},
url = {https://www.sciencedirect.com/science/article/pii/000187089290028J},
author = {Andreas W.M Dress and Walter Wenzel}
}

@article{BlandVergnas78,
title = {Orientability of matroids},
journal = {Journal of Combinatorial Theory, Series B},
volume = {24},
number = {1},
pages = {94-123},
year = {1978},
issn = {0095-8956},
doi = {https://doi.org/10.1016/0095-8956(78)90080-1},
url = {https://www.sciencedirect.com/science/article/pii/0095895678900801},
author = {Robert G Bland and Michel {Las Vergnas}}
}

@article{Anderson19,
title = {Vectors of matroids over tracts},
journal = {Journal of Combinatorial Theory, Series A},
volume = {161},
pages = {236-270},
year = {2019},
issn = {0097-3165},
doi = {https://doi.org/10.1016/j.jcta.2018.08.002},
url = {https://www.sciencedirect.com/science/article/pii/S0097316518301109},
author = {Laura Anderson}
}

@article{HeunenPatta18, title={The Category of Matroids}, volume={26}, DOI={https://doi.org/10.1007/s10485-017-9490-2}, number={3}, journal={Applied Categorical Structures}, author={Heunen, Chris and Patta, Vaia}, year={2018}, pages={205–237}}

@book{White86, place={Cambridge}, series={Encyclopedia of Mathematics and its Applications}, title={Theory of Matroids}, editor={White, Neil}, publisher={Cambridge University Press}, year={1986}, collection={Encyclopedia of Mathematics and its Applications}}

@phdthesis{Eppolito22,
    author = {Eppolito, Chris},
    title = {Matroids: Mystic Monoliths, Meta Missiles, and Myopic Meadows},
    school = {State University of New York at Binghamton},
    year = {2022}
}

@misc{JunSistkoWright25,
      title={Proto-Exact Categories of Matroids over Idylls and Tropical Toric Reflexive Sheaves}, 
      author={Jaiung Jun and Alex Sistko and Cameron Wright},
      year={2025},
      eprint={2509.08144},
      archivePrefix={arXiv},
      primaryClass={math.CT},
      url={https://arxiv.org/abs/2509.08144}, 
}

@misc{JarraLorscheidVital26,
      title={Quiver matroids -- Matroid morphisms, quiver Grassmannians, their Euler characteristics and $\mathbb{F}_1$-points}, 
      author={Manoel Jarra and Oliver Lorscheid and Eduardo Vital},
      year={2026},
      eprint={2404.09255},
      archivePrefix={arXiv},
      primaryClass={math.CO},
      url={https://arxiv.org/abs/2404.09255}, 
}

@Inbook{DyckerhoffKapranov19,
    author="Dyckerhoff, Tobias
    and Kapranov, Mikhail",
    title="Topological 1-Segal and 2-Segal Spaces",
    bookTitle="Higher Segal Spaces",
    year="2019",
    publisher="Springer International Publishing",
    address="Cham",
    pages="9--30",
    isbn="978-3-030-27124-4",
    doi="10.1007/978-3-030-27124-4_2",
    url="https://doi.org/10.1007/978-3-030-27124-4_2"
}

@article{EppolitoJunSzczensny20, title={Proto-exact categories of matroids, Hall algebras, and K-theory}, volume={296}, DOI={https://doi.org/10.1007/s00209-019-02429-z}, number={1}, journal={Mathematische Zeitschrift}, author={Eppolito, Chris and Jun, Jaiung and Szczesny, Matt}, year={2020}, pages={147–167}}

@misc{Hamada26,
      title={On the category of modules over bands: relative schemes, hyperring schemes and proto-exactness}, 
      author={Lucas Hamada},
      year={2026},
      eprint={2606.01093},
      archivePrefix={arXiv},
      primaryClass={math.AG},
      url={https://arxiv.org/abs/2606.01093}, 
}
\end{document}